\documentclass[notitlepage,11pt,reqno]{amsart}

\usepackage{amsopn,esint,nicefrac}
\usepackage{mathabx}
\usepackage[final]{hyperref}
\usepackage[T1]{fontenc}
\usepackage{graphicx}
\usepackage{cancel,pdfsync}
\usepackage{bbm}
\usepackage{tikz}
\usetikzlibrary{patterns} 
\usepackage{float}

\usepackage[utf8]{inputenc}
\usepackage{apptools}
\usepackage[margin=1in]{geometry}
\allowdisplaybreaks
\newcommand{\stkout}[1]{\ifmmode\text{\sout{\ensuremath{#1}}}\else\sout{#1}\fi}

\renewcommand{\fint}{\mathop{\int\hskip -0,93em -\, \!\!\!}}
\newcommand{\smallfint}{\mathop{\int\hskip -0,78em -\, \!\!}}

\newcommand{\R}{\mathbb{R}}

\newcommand{\Rn}{\mathbb{R}^n}
\usepackage{graphicx,enumitem,dsfont,upgreek}
 \usepackage[dvips]{epsfig}
 \usepackage[mathscr]{eucal}
\usepackage{amscd}
\usepackage{amssymb,nicefrac}
\usepackage{amsthm}
\usepackage{amsmath}
\usepackage{latexsym}
\usepackage{dsfont}
\usepackage{upref}
\usepackage{hyperref}

\usepackage{color}
\theoremstyle{plain}

\newtheorem{thm}{Theorem}[section]
\theoremstyle{plain}
\newtheorem{lem}[thm]{Lemma}

\newtheorem{cor}[thm]{Corollary}

\newtheorem{defi}[thm]{Definition}
\newtheorem{rem}{Remark}[section]
\theoremstyle{definition}
\newtheorem*{maintheorem*}{Main Theorem}
\newtheorem*{maincorollary*}{Main Corollary}
{%
\setcounter{enumi}{0}

\begin{enumerate}}%
{\end{enumerate} }

{%
\setcounter{enumi}{0}

\begin{enumerate}}%
{\end{enumerate} }

\newcommand{\norm}[1]{\ensuremath{\left\|#1\right\|}}

\newcommand{\cone}{\ensuremath{\mathcal{C}}}

\newcommand{\cK}{\ensuremath{\mathscr{K}}}

\newcommand{\cL}{\ensuremath{\mathcal{L}}}
\newcommand{\cM}{\ensuremath{\mathcal{M}}}

\newcommand{\dist}{{\rm dist}}

\newcommand{\tail}{\texttt{Tail}}
\newcommand{\grad}{\nabla}
\newcommand{\bu}{\bar{u}}
\newcommand{\bX}{\ensuremath{\mathbb{X}_0}}

\newcommand{\dx}{\ensuremath{\, {\rm d}x}}
\newcommand{\dy}{\ensuremath{\, {\rm d}y}}
\newcommand{\dz}{\ensuremath{\, {\rm d}z}}

\newcommand{\dr}{\ensuremath{\, {\rm d}r}}
\newcommand{\dt}{\ensuremath{\, {\rm d}t}}
\newcommand{\dnu}{\ensuremath{\, {\rm d}\nu}}

\newcommand{\supp}{\ensuremath{\mathrm{supp}\,}}

\numberwithin{equation}{section} \allowdisplaybreaks

\DeclareMathOperator*{\Argmin}{arg\,min}

\title[Strong comparison principle and symmetry]{Strong comparison principle and symmetry results for the fractional $p$-Laplacian}

\begin{document}

\author{Anup Biswas}

\author{Subhajit Roy}

\author{Aniket Sen}

\address{Indian Institute of Science Education and Research-Pune, Dr.\ Homi Bhabha Road, Pashan, Pune 411008, INDIA.}
\email{ anup@iiserpune.ac.in, subhajit.roy@iiserpune.ac.in, aniket.sen@students.iiserpune.ac.in}

\begin{abstract}
In this article, we study the equation
$$
(-\Delta_p)^s u = f(u)
$$
in a bounded domain $\Omega\subset \Rn$, where $n\geq 2$, $p>2$, and $f$ is locally Lipschitz. We establish a strong comparison principle in a fairly general setting and use it to derive symmetry results for positive $C^1$ solutions satisfying Dirichlet boundary conditions. We also show that the $C^1$ regularity assumption is indeed satisfied for
$p\in \left[2,\frac{2}{1-s}\right)$. 
\end{abstract}


\maketitle


\section{Introduction}
The main goal of this article is to study the strong maximum principle and symmetry properties of solutions to
\begin{equation}\label{E1.1}
(-\Delta_p)^s u  = f(u)\quad \text{in}\; \Omega.
\end{equation}
Here, $\Omega$ is an open set in $\Rn$ with $n\geq 2$, and $f:\R\to\R$ is a locally Lipschitz function. We investigate these symmetry properties for positive solutions satisfying the Dirichlet boundary condition in $\Omega^c$. The fractional $p$-Laplacian operator $(-\Delta_p)^s$ is defined as
$$(-\Delta_p)^s u(x)=\mathrm{pv}\int_{\Rn} J_p(u(x)-u(y)) \frac{\dy}{|x-y|^{n+sp}}\, ,$$
where $s\in (0, 1)$, $p\in (1, \infty)$, and $J_p(t)=|t|^{p-2}t$. 

Throughout this paper, we consider both weak and local weak solutions. To formalize these concepts, we introduce the function space
$$\bX(\Omega)=\{v\in W^{s, p}(\Rn)\; :\; v=0\quad \text{in}\; \Omega^c\}.$$
We denote by $W^{s, p}_0(\Omega)$ the completion of $C^\infty_c(\Omega)$ with respect to the norm $\norm{\cdot}_{W^{s, p}(\Omega)}$. A function $u\in W^{s,p}(\Omega_1)\cap L^{p-1}_{sp}(\Rn)$, where $\Omega_1\Supset\Omega$, is called a \textit{weak solution} to \eqref{E1.1} if 
$$
\frac{1}{2}\int_{\Rn}\int_{\Rn} J_p(u(x)-u(y))(\phi(x)-\phi(y))\frac{\dx\dy}{|x-y|^{n+sp}}=\int_{\Rn} f(u(x))\phi(x)\dx
$$
holds for all $\phi\in \bX(\Omega)$. Here, $L^{q}_\alpha(\Rn)$ denotes the tail space defined by
\begin{align*}
L^{q}_{\alpha}(\Rn)  &= \left\{v\in L^{q}_{\rm loc}(\Rn)\;: \; \int_{\Rn} \frac{|v(z)|^{q}}{1+|z|^{n+\alpha}}\dz<\infty\right\},
\\
\tail_{\alpha, q}(x, r; u)&=\left(\int_{B_r^c(x)} \frac{|u(z)|^{q}}{|z-x|^{n+\alpha}}\dz\right)^{\frac{1}{q}}.
\end{align*}
Such weak solutions are typically obtained by minimizing the energy functional 
$$\mathcal{E}(u) = \frac{1}{2p}\int_{\Rn}\int_{\Rn} |u(x)-u(y)|^p\frac{\dx\dy}{|x-y|^{n+sp}} -\int_{\Omega} F(u(x))\dx$$
over the class of functions in $W^{s,p}(\Rn)$ with prescribed values in $\Omega^c$, where $F(t)=\int_0^t f(\tau) d\tau$. 

To define local weak solutions to \eqref{E1.1}, we require the local Sobolev space
\begin{align*}
W^{s, p}_{\rm loc}(\Omega)&=\{v\in L^{p}_{\rm loc}(\Omega)\; :\; v\in W^{s,p}(\Omega^\prime)\; \text{for all}\; \Omega^\prime\Subset\Omega\}.
\end{align*}
A function $u\in W^{s, p}_{\rm loc}(\Omega)\cap L^{p-1}_{sp}(\Rn)$ is said to be a \textit{local weak solution} to \eqref{E1.1} if 
$$
\frac{1}{2}\int_{\Rn}\int_{\Rn} J_p(u(x)-u(y))(\phi(x)-\phi(y))\frac{\dx\dy}{|x-y|^{n+sp}}=\int_{\Rn} f(u(x))\phi(x)\dx
$$
holds for all $\phi\in C^\infty_c(\Omega)$. It is clear that every weak solution is also a local weak solution. Furthermore, for local weak solutions, the class of test functions can be extended to $\bX(B)$ for any $B\Subset\Omega$ by virtue of \cite[Theorem 6.62]{Leoni}.

The bulk of our analysis is sustained by local weak solutions. However, the moving plane machinery in Section~\ref{S-moving} explicitly requires weak solutions to leverage Hopf's lemma, which depends fundamentally on the validity of the weak comparison principle.

As mentioned previously, one of the primary goals of this article is to investigate the strong comparison principle (SCP) associated with equation \eqref{E1.1}. That is, if $u$ is a subsolution to \eqref{E1.1} and $v$ is a supersolution to \eqref{E1.1} satisfying $v \geq u$ in $\mathbb{R}^n$, does 
$$
\{u=v\} \cap \Omega \neq \emptyset \implies u \equiv v \text{ in } \mathbb{R}^n?
$$

We begin by reviewing the SCP results known for the classical $p$-Laplace operator. In \cite[Theorem~1.4]{D98}, Damascelli proved that given $u, v \in C^1(\Omega)$ with $v \geq u$ in $\Omega$, the SCP holds in any connected component of $\Omega \setminus \mathcal{Z}_{u,v}$, where 
$
\mathcal{Z}_{u,v} = \Omega \cap \{\nabla u = 0=\nabla v\}.
$
Assuming $f$ is nondecreasing and $u < v$ on $\partial\Omega$, the SCP was established by Guedda and V\'{e}ron in \cite{GV89}. Further improvements in this direction were made by Lucia and Prashanth \cite{LP04}, as well as Roselli and Sciunzi \cite{RS07}. We also mention the work of Cuesta and Tak\'{a}\v{c} \cite{CT98}, where the SCP was investigated for nonlinearities that do not change sign.

On the other hand, the theory of the SCP in the context of the fractional $p$-Laplacian is rather limited. For $p=2$, one can use linearity to derive the SCP; see, for instance, Caffarelli and Silvestre \cite{CS09}. To the best of our knowledge, for general values of $p$, the only available results on the strong comparison principle are due to Jarohs \cite{SJ18} and Colasuonno et al. \cite{CFGQ}, where the nonlinearities considered are of homogeneous type. Moreover, the SCP established in \cite{SJ18} requires the condition
$p-1>sp$,
while the result in \cite{CFGQ} assumes \emph{a priori} an additional integral condition (see Theorem 2.1 therein), which appears difficult to verify unless
$p-1>sp$.
Let us also mention that for $p \geq 2$, if we assume $u, v \in C^{1,1}(\Omega)$, then $u$ and $v$ become classical sub- and supersolutions, making the establishment of the SCP rather trivial; see, for instance, \cite{CL18}. As one might expect, the regularity of the sub- and supersolutions becomes quite important in the investigation of the SCP. It is also not practical to expect solutions to be in $C^{1,1}$ given recent developments in the regularity theory of fractional $p$-Laplacian operators. In fact, assuming $f$ is Lipschitz, we may expect the solution to be in $C^{0, \gamma_0}_{\rm loc}(\Omega)$ (loosely speaking), where 
\[
\gamma_0 = \begin{cases}
\min\left\{1, \frac{sp}{p-2}\right\} & \text{if } p > 2, 
\\[2mm]
1 & \text{if } p \in (1,2].
\end{cases}
\]
See, for instance, \cite{BDLMS24a, BDLMS24b, BLS18, BS25, BT25, GL24} for interior regularity, \cite{IMS16,IMS20} for boundary regularity, and \cite{KMS15a,KMS15b} for fine zero-order regularity estimates. The recent development by Giovagnoli, Jesus, and Silvestre \cite{GJS25} suggests that the solutions are in $C^{1,\alpha}_{\rm loc}$ for $p \in \left[2, \frac{2}{1-s}\right)$ (though their proof considers $f=0$, it can be easily extended to locally Lipschitz data; see Theorem~\ref{T5.3} below). One of the main contributions of this article is to establish the SCP keeping this regularity in mind. In particular, we prove the SCP for $p>2$ in the following two situations:
\begin{itemize}
\item $\frac{sp}{p-1} < 1$, where both $u$ and $v$ are locally $\left(\frac{sp}{p-1}+\epsilon\right)$-H\"{o}lder continuous in $\Omega$ for some $\epsilon>0$.
\item $\frac{sp}{p-2} > 1$, where both $u$ and $v$ are locally Lipschitz continuous in $\Omega$ with $u \in C^1(\Omega)$. Moreover, $\{\nabla u=0\} \cap \Omega \Subset \Omega$.
\end{itemize}
See Theorem~\ref{T2.8} below for a detailed statement. We allow the source term to depend on $x$, and $f$ need not be monotone. Also, note that the $C^1$ assumption is not required when $p-2 \leq sp < p-1$. As far as the proof is concerned, we borrow inspiration from \cite[Theorem~1.4]{D98}. However, due to the nonlocal nature of the problem, the present situation turns out to be much more complicated. We draw our inspiration from
the estimates of \cite{DKP14, DKP16} which are essential for the nonlocal Harnack inequality in those works. Part of the reason we restrict ourselves to the case $p>2$ is the nonlocal tail; see, for instance, the proofs of Lemmas~\ref{L2.6} and~\ref{L2.7}.

The second contribution of this article is the study of the symmetry of positive solutions to \eqref{E1.1}. Let us first briefly review the established developments for the classical $p$-Laplacian operator. 
For $p=2$, the symmetry result was proved in the seminal work of Gidas, Ni, and Nirenberg \cite{GNN}, who exploited the Alexandrov--Serrin moving plane method \cite{Serrin71}. One  of the important outcomes of this moving plane method is the strictly monotonicity property of the solution.
For $p \neq 2$, however, the operator becomes singular or degenerate, making the availability of an appropriate SCP highly challenging. Assuming $\Omega$ is a ball, the symmetry of the solution was first obtained by Badiale and Nabana \cite{BN94} under the hypothesis that the gradient of the solution vanishes only at the center of the ball. Later, for $1 < p < 2$, symmetry results were established in the works of Damascelli and Pacella \cite{DP98, DP00} without imposing any assumptions on the critical set. For $p > 2$, a similar symmetry property was established by Damascelli and Sciunzi \cite{DS04, DS06} under the hypothesis that $f$ is positive in $(0, \infty)$. It is necessary to point out that there exist sign-changing Lipschitz functions $f$ in a ball for which one can obtain non-radial positive solutions (see \cite{Br00,KS90}). Moreover, for $p > 2$, there are positive radial solutions that are not strictly decreasing along the radius \cite[Example~5.1]{GKPR}.

On the other hand, analogous results for nonlocal operators are relatively scarce. For the fractional Laplacian, symmetry and strict monotonicity via the moving plane method were established by Felmer and Wang \cite{FW14}; see also \cite{BMS18,BM25}. In the case $p\neq 2$, there are some existing works, most of which are based on the {\it direct method} of moving plane introduced by Chen and Li \cite{CL18}. A common assumption in these works is that
$$
u\in C^{1,1}(\Omega),
$$
which ensures that the operator is classically well defined for $p\geq 2$. However, even in comparison with the known optimal regularity theory for the local $p$-Laplace equation \cite{IM89}, such a regularity assumption appears rather restrictive.

Motivated by the recent regularity developments in \cite{GJS25}, we work instead with $C^1$ solutions, and this assumption is justified by Theorem~\ref{T5.3} below. Moreover, we allow for a very general nonlinearity $f$ satisfying only
\[
f(0)\geq 0.
\]
This condition is significantly weaker than the commonly imposed assumption in the local case that $f$ is positive on $(0,\infty)$. Part of the reason this weaker hypothesis suffices is that, due to the nonlocal nature of the operator, a local strong comparison principle implies a global one.

Moreover, a closer examination of the proofs in \cite{CL18} reveals that the key boundary point estimate established in \cite[Theorem~2.3]{CL18} may not be sufficient to complete the argument of \cite[Theorem~3.1]{CL18}; see Remark~\ref{R5.1} below for further discussion. In the present work, we assume that $\Omega$ is strictly convex and, with the aid of the Hopf lemma, we are able to complete the proof. Another crucial ingredient in our arguments is a strong comparison principle for reflected functions, which is developed in Section~\ref{S-anti}.

The remainder of the article is organized as follows. In the next section, we collect several preliminary results concerning viscosity solutions that will play a crucial role in our arguments. In Section~\ref{S-SCP}, we establish the strong comparison principle, while Section~\ref{S-anti} is devoted to the proof of the strong comparison principle for  functions reflected with respect to a hyperplane. Section~\ref{S-moving} contains the proof of our symmetry results. Finally, in the Appendix, we provide a sketch of the $C^{1,\alpha}$ regularity result in the presence of Lipschitz continuous data.

\section{Preliminary results for viscosity solutions}\label{S-prelim}
The equivalence between local weak solutions and viscosity solutions is by now well known; see \cite{KKL19,BM21}. We shall exploit this equivalence in our application of the moving plane method. In fact, we only use the easier direction of the equivalence, namely that every continuous local weak solution is also a viscosity solution. One reason for doing so is that it allows us to adapt certain ideas from \cite{CL18}, which were developed for classical $C^{1,1}$ solutions.

To introduce the definition of viscosity solutions, we first recall some notation from \cite{KKL19}. Since it was shown in \cite{KKL19} that $(-\Delta_p)^s$ may fail to be classically well-defined even for certain $C^2$ functions, one must restrict attention to an appropriate subclass of test functions in the definition of viscosity solutions.
Given an open set $D$, we denote by $C^2_\eta(D)$, a subset of $C^2(D)$, defined as
$$
C^2_\eta(D)=\left\{\phi\in C^2(D)\; :\; \sup_{x\in D}\left[\frac{\min\{d_\phi(x), 1\}^{\eta-1}}{|\nabla\phi(x)|} +
\frac{|D^2\phi(x)|}{(d_\phi(x))^{\eta-2}}\right]<\infty\right\},
$$
where
$$ 
d_\phi(x)=\dist(x, N_\phi)\quad \text{and}\quad N_\phi=\{x\in D\; :\; \nabla\phi(x)=0\}.$$

The above restricted class of test functions becomes necessary to define $(-\Delta_p)^s$ in the classical sense in the singular case, that is, for
$p\leq \frac{2}{2-s}$. 
Now we are ready to define the viscosity solution from \cite[Definition~3]{KKL19}. 

\begin{defi}\label{Defi-vis}
A function $u:\Rn\to \R$ is a viscosity subsolution (supersolution) to $(-\Delta_p)^s u= f$ in $\Omega$ if it satisfies the following
\begin{itemize}
\item[(i)] $u$ is upper (lower) semicontinuous in $\bar\Omega$.
\item[(ii)] If $\varphi\in C^2(B_r(x_0))$ for some $B_r(x_0)\subset \Omega$ satisfies $\varphi(x_0)=u(x_0)$,
$\varphi\geq u$ ($\varphi\leq u$) in $B_r(x_0)$ and one of the following holds
\begin{itemize}
\item[(a)] $p>\frac{2}{2-s}$ or $\nabla\varphi(x_0)\neq 0$,

\item[(b)] $p\leq \frac{2}{2-s}$ and $\nabla\varphi(x_0)= 0$ is such that $x_0$ is an isolated critical point of $\varphi$, and
$\varphi\in C^2_\eta(B_r(x_0))$ for some $\eta>\frac{sp}{p-1}$,
\end{itemize}
then we have 
$$(-\Delta_p)^s\varphi_r(x_0)  \leq f(x_0) \quad \left((-\Delta_p)^s \phi_r(x_0) \geq f(x_0)\right),$$
 where
\[
\varphi_r(x)=\left\{\begin{array}{ll}
\varphi(x) & \text{for}\; x\in B_r(x_0),
\\[2mm]
u(x) & \text{otherwise}.
\end{array}
\right.
\]

\item[(iii)] $u_+\in L^{p-1}_{sp}(\Rn)$ ($u_-\in L^{p-1}_{sp}(\Rn)$, respectively).
\end{itemize}
A  viscosity solution of $(-\Delta_p)^s u= f$ in $\Omega$ is both sub and supersolution in $\Omega$. 
\end{defi}

Let us now introduce the inf and sup convolutions which we apply on the viscosity super and subsolutions, respectively.
Given a bounded upper-semicontinuous function $u$ the sup-convolution $u^{\varepsilon}$ is given by
$$u^{\varepsilon} (x) = \sup_{y \in \Rn} \left\{u(x+y) - \frac{|y|^2}{\varepsilon}\right\} = \sup_{y \in \Rn} 
\left\{u(y) - \frac{|x-y|^2}{\varepsilon}\right\} =  u(x^*) - \frac{|x-x^*|^2}{\varepsilon} .$$
Similarly, for a bounded and lower-semicontinuous function $v$, the inf-convolution $v_\varepsilon$ is given by
$$v_{\varepsilon} (x) = \inf_{y \in \Rn} \left\{v(x+y) + \frac{|y|^2}{\varepsilon}\right\} = \inf_{y \in \Rn} \left\{v(y) + \frac{|x-y|^2}{\varepsilon}\right\} =  v(x_*) + \frac{|x-x_*|^2}{\varepsilon} .$$
The following result, which is known to be standard for translation invariant operators, will be key to our approach.

\begin{lem}\label{PL2.2}
Let $p>\frac{2}{2-s}$.
Let $\Omega$ be an open set and $f\in C(\Omega)$. Let $u\in USC(\bar\Omega)\cap L^\infty(\Rn)$  is a viscosity solution to $(-\Delta_p)^s u \leq f$
in $\Omega$. Then for $\Omega_1\Subset\Omega$, we have  $(-\Delta_p)^s u^\varepsilon \leq f +d_\varepsilon$ in $\Omega_1$, where 
the function $d_\varepsilon$ depends on $\Omega_1$ and goes to
$0$, uniformly over $\Omega_1$, as $\varepsilon \to 0$. An analogous result holds for supersolutions.
\end{lem}

\begin{proof}
Pick any  $x_0 \in \Omega_1$. Let $\varphi$ be a test function that touches $u^{\varepsilon}$ from above at $x_0$ in some neighbourhood $B_r(x_0)$.
Without loss of generality, we may assume that $r<\frac{1}{2}\dist(\Omega_1, \partial\Omega)$. Define 
$$Q(x):= \varphi (x+ x_0 -x_0^*) + \frac{1}{\varepsilon} |x_0 - x_0^*|^2.$$
From the definition of $u^{\varepsilon}$ observe that $\frac{|x_0 - x_0^*|^2}{\varepsilon} \leq {2 \Vert u \Vert_\infty}$, therefore we can pick $\varepsilon_1$ such that for all $\varepsilon \leq \varepsilon_1$ and $x_0 \in \Omega_1$, we have $|x_0-x_0^*|<\frac{1}{4}\dist(\Omega_1, \partial\Omega)$, implying
$x_0^* \in \Omega$. From the definition of sup-convolution it follows that $u(x) \leq u^\varepsilon (x+x_0-x_0^*) + \frac{|x_0 - x_0^*|^2}{\varepsilon}$. Then, for $|x- x_0^*| < r$ (which implies $x+x_0-x_0^*\in\Omega$) 
$$u(x) \leq \varphi(x+x_0-x_0^*) +\frac{1}{\varepsilon}|x_0 - x_0^*|^2 = Q(x)$$ and $u(x_0^*) = Q(x_0^*)$. Define
$$w(x) := \begin{cases} Q(x) \quad \text{if } x\in B_r(x_0^*),
\\[2mm]
 u(x) \quad \text{otherwise.}
 \end{cases} \quad \text {and} \quad
  \varphi_r(x) := \begin{cases}
   \varphi(x) \quad \text{if } x\in B_r(x_0),
   \\[2mm]
   u^\varepsilon(x) \quad \text{otherwise.} 
  \end{cases}$$ 
Observe that 
$ Q(x_0^*) - Q(x_0^*+z)= \varphi(x_0) - \varphi(x_0 + z)$, and $u(x_0^*) - u(x_0^* + z) \geq u^\varepsilon(x_0) - u^\varepsilon(x_0 + z)$. 
Hence $w(x_0^*)-w(x^*_0+z)\geq \varphi_r(x_0)-\varphi_r(x_0+z)$ for all $z$.
By the definition of viscosity subsolution
\begin{align*}
(-\Delta_p)^s w(x_0^*) &\leq f(x_0^*),
\\
\Rightarrow {\rm pv}\int_{\Rn}J_p[ w(x_0^*) - w(x_0^*+z)]\frac{\dz}{|z|^{n+sp}} &\leq f(x_0) -f(x_0) +f(x_0^*),
\\
\Rightarrow {\rm pv} \int_{\Rn}J_p[ \varphi_r(x_0) - \varphi_r(x_0+z)]\frac{\dz}{|z|^{n+sp}} &\leq f(x_0) +|f(x_0) -f(x_0^*)|,
\\
\Rightarrow (-\Delta_p)^s \varphi_r (x_0) &\leq f(x_0) + d_\varepsilon,
\end{align*}
where  $d_\varepsilon(x) := \sup_{|x-y|\leq \sqrt{2\norm{u}_\infty\varepsilon}}|f(x)-f(y)|$. Hence the result.
\end{proof}

Recall that a function $u$ is said to be $C^{1,1}$ at a point $x_0$, denoted by $u\in C^{1,1}(x_0)$, if there exist a vector $\xi\in\Rn$ and 
constants $M, r$ satisfying
$$ |u(x_0+y)-u(x_0)-\xi \cdot y|\leq M|y|^2\quad \text{for}\; y\in B_r(0).$$
It is evident that $u\in C^{1,1}(x_0)$ implies $\xi=\grad u(x_0)$.
Again, since
$$u^\varepsilon(x)+\frac{1}{\varepsilon}|x|^2=\sup_{y\in\Rn}\left\{u(y)+\frac{2}{\varepsilon}x\cdot y -\frac{1}{\varepsilon}|y|^2\right\},$$
we have $u^\varepsilon$ semiconvex and by an analogous reason, $v_\varepsilon$ semiconcave. The following result will be applied on
$u^\varepsilon$ and $v_\varepsilon$.

\begin{lem}\label{PL2.3}
Let $v$ be semiconcave and $u$ be semiconvex in a ball $B_r(x_0)$, respectively. Let $x_0$ be a point of local minima of $w=v-u$ in
$B_r(x_0)$. Then,
\begin{itemize}
\item[(i)] $w$ is $C^{1,1}$ at $x_0$.
\item[(ii)] $v, u$ both are $C^{1,1}$ at $x_0$.
\end{itemize}
\end{lem}

\begin{proof}
Note that $-u$ is semiconcave, and therefore, $w$ is semiconcave. Also, after adding a suitable quadratic function, we can assume $w, v$ and
$-u$ are concave. Therefore, $w, v$ and $-u$ can be touched from above by a paraboloid at the point $x_0$. Since $w$ has a minimum at the point $x_0$, it can also be touched by a paraboloid from below at the point $x_0$. Therefore, $w\in C^{1,1}(x_0)$ which proves (i).

Again, by \cite[Theorem~{23.8}]{Rock}, $\grad w(x_0)=\partial(w)(x_0)=\partial(v)(x_0)+\partial(-u)(x_0)$. Therefore, $\partial(v)(x_0)$ and 
$\partial(-u)(x_0)$ are singleton, implying the $v, u$ are differentiable at $x_0$ \cite[Theorem~25.1]{Rock}. Hence $\grad w(x_0)=\grad v(x_0)-\grad u(x_0)$. Since, $v(x_0+y)-v(x_0)-\grad v(x_0)\cdot y\leq 0$ and $-u(x_0+y)+u(x_0)+\grad u(x_0)\cdot y\leq 0$, we get
$$
|v(x_0+y)-v(x_0)-\grad v(x_0)\cdot y|+|-u(x_0+y)+u(x_0)+\grad u(x_0)|
=|w(x_0+y)-w(x_0)-\grad w(x_0)\cdot y|
$$
and using the fact $w\in C^{1,1}(x_0)$, it follows that $v, u\in C^{1,1}(x_0)$. This gives us (ii) and the proof is done.
\end{proof}

The benefit of being in $C^{1,1}(x_0)$ is that the solution becomes classical at $x_0$. This is the content of our next lemma.
\begin{lem}\label{PL2.4}
Let $p >\frac{2}{2-s}$ and $(-\Delta_p)^s u\leq f$ in the viscosity sense.
Suppose that $u \in C^{1,1}(x_0) \cap L_{sp}^{p-1}(\Rn)$, then $(-\Delta_p)^s u$ is classically defined at $x_0$ and
$(-\Delta_p)^s u(x_0)\leq f(x_0)$.
\end{lem}

\begin{proof}
By the definition of $C^{1,1}(x_0)$, we have
$$|u(x_0+y)-u(x_0)-\grad u(x_0)\cdot y|\leq M|y|^2 \quad \text{for all}\; y\in B_r(0),$$
for some $M, r>0$. Note that, defining $\varphi(y):= u(x_0)+\grad u(x_0)\cdot (y-x_0)+M|y-x_0|^2$, $\varphi$ touches $u$ at $x_0$ from above 
in $B_r(x_0)$. Therefore, for $\delta<r$, if we define
\[
\varphi_\delta(x)=\left\{
\begin{array}{ll}
\varphi(x) & \text{for}\; x\in B_\delta(x_0),
\\[2mm]
u(x) & \text{otherwise},
\end{array}
\right.
\] 
from the definition of viscosity solution it then follows that $(-\Delta_p)^s \varphi_\delta(x_0)\leq f(x_0)$. We show that 
$\lim_{\delta\to 0}(-\Delta_p)^s \varphi_\delta(x_0)=(-\Delta_p)^s u(x_0)$ which will complete the proof. Since $\varphi_\delta$ and 
$u$ are in $C^{1,1}(x_0)$ with the same $M, r$ as above and $\grad u(x_0)=\grad \varphi_\delta(x_0)$, it is enough to show that for any function $\zeta$ satisfying
\begin{equation}\label{PL2.4A}
|\zeta(x_0+y)-\zeta(x_0)- \xi\cdot y|\leq M|y|^2 \quad \text{for all}\; y\in B_r(0),
\end{equation}
we have
\begin{equation}\label{PL2.4B}
\lim_{\varepsilon\to 0} {\rm pv}\int_{B_\varepsilon} J_p(\zeta(x_0)-\zeta(x_0+z))\frac{\dz}{|z|^{n+sp}}=0,
\end{equation}
where the limit is uniform across all $\zeta$ satisfying \eqref{PL2.4A} with the same $(\xi, M, r)$. 

\eqref{PL2.4B} is quite standard, see for instance, \cite{KKL19}. We add a proof for the convenience. 
For $\xi=0$, the left-hand side of \eqref{PL2.4B} can be estimates as
$${\rm pv}\int_{B_\varepsilon} J_p(\zeta(x_0)-\zeta(x_0+z))\frac{\dz}{|z|^{n+sp}}
\leq M^{p-1}\int_{B_\varepsilon} |z|^{2(p-1)}\frac{\dz}{|z|^{n+sp}}\leq C M^{p-1}\varepsilon^{2(p-1)-sp}.$$
The result follows since $p>\frac{2}{2-s}$. So we assume $\xi\neq 0$.
Let $\ell(y):=\zeta(x_0)+\xi\cdot(y-x_0)$. By symmetry, ${\rm pv}\int_{B_\varepsilon}J_p(\ell(x_0)-\ell(x_0+z))\frac{dy}{|z|^{n+sp}} = 0$. 
Hence
\begin{align*}
&\left|{\rm pv} \int_{B_\varepsilon}J_p(\zeta(x_0)-\zeta(x_0+z))\frac{\dz}{|z|^{n+sp}}\right| 
\\
&\quad\leq \int_{B_\varepsilon}\left|J_p(\zeta(x_0)-\zeta(x_0+z))-J_p(\ell(x_0)-\ell(x_0+z))\right|\frac{\dz}{|z|^{n+sp}} 
\\
&\quad \leq c \int_{B_\varepsilon}\big(|\xi \cdot z|+|\zeta(x_0+z)-\ell(x_0+z)|\big)^{p-2}|\zeta(x_0+z)-\ell(x_0+z)|\frac{\dz}{|z|^{n+sp}} 
\\
&\quad \leq c M\int_{B_\varepsilon}\big(|\xi \cdot z|+|\zeta(x_0+z)-\ell(x_0+z)|\big)^{p-2} |z|^2 \frac{\dz}{|z|^{n+sp}} \,.
\end{align*}
For $p\geq 2$, we have from \eqref{PL2.4A} that
$$\big(|\xi \cdot z|+|\zeta(x_0+z)-\ell(x_0+z)|\big)^{p-2}\leq (\max\{|\xi|, M\} |z|)^{p-2},$$
giving us 
$$\left|{\rm pv} \int_{B_\varepsilon(x)}J_p(\zeta(x_0)-\zeta(x_0+z))\frac{\dz}{|z|^{n+sp}}\right|\leq C \varepsilon^{p(1-s)}.$$
For $p<2$, we compute the integral as follows
\begin{align*}
\int_{B_\varepsilon}\big(|\xi \cdot z|+|\zeta(x_0+z)-\ell(x_0+z)|\big)^{p-2} |z|^2 \frac{\dz}{|z|^{n+sp}}&
\leq \int_{B_\varepsilon}|\xi \cdot z|^{p-2} |z|^2 \frac{\dz}{|z|^{n+sp}}
\\
&=\int_{\mathbb{S}^n} |\xi\cdot \omega|^{p-2} {\rm d}\omega \int_0^\varepsilon r^{p(1-s)-1} {\rm d}r
\\
&= \frac{\varepsilon^{p(1-s)}}{p(1-s)} \int_{\mathbb{S}^n} |\xi\cdot \omega|^{p-2} {\rm d}\omega.
\end{align*}
Hence we have \eqref{PL2.4B}.
\end{proof}

Our last result of this section identifies a range of parameters for which classical subsolution are possible.

\begin{lem}\label{PL2.5}
Let $\frac{sp}{p-1}<1$ and $(-\Delta_p)^s u\leq f$ in $\Omega$ in the (local) weak sense where $f\in L^1(\Omega)$. If for some $\epsilon>0$ we have 
$u\in C^{\frac{sp}{p-1}+\epsilon}_{\rm loc}(\Omega)$, then $(-\Delta_p)^s u$ is classically defined and we have
$(-\Delta_p)^s u\leq f$ satisfied almost everywhere.
\end{lem}

\begin{proof}
Without any loss of generality, we may assume that $\frac{sp}{p-1}+\epsilon\leq 1$.
Note that, since $u\in C^{\frac{sp}{p-1}+\epsilon}_{\rm loc}(\Omega)$, 
$$
\frac{|u(x)-u(y)|^{p-1}}{|x-y|^{n+sp}}\lesssim |x-y|^{\epsilon(p-1)-n}.
$$
Therefore, the map 
$$\Omega\ni x\mapsto \int_{\Rn} J_p(u(x)-u(y))\frac{\dy}{|x-y|^{n+sp}}$$
is continuous in $\Omega$. Thus, for any compact set $\Omega_1\Subset\Omega$, and for nonnegative $\varphi\in C^\infty_c(\Omega_1)$ we have
\begin{align*}
\int_{\Rn} f\varphi\dx &\geq \frac{1}{2}\int_{\Rn}\int_{\Rn} J_p(u(x)-u(y)) (\varphi(x)-\varphi(y))\frac{\dx\dy}{|x-y|^{n+sp}}
\\
&=\int_{\Rn} \varphi(x)(-\Delta_p)^s u(x)\dx.
\end{align*}
Now the result follows from the density argument.
\end{proof}


\section{Strong Comparison Principle}\label{S-SCP}
In this section, we consider local weak solutions to the equation
\begin{equation}\label{main-sc}
(-\Delta_p)^s u = f(x, u) \quad \text{in}\; \Omega.
\end{equation}
We assume that $f:\Omega\times \R\to\R$ is continuous and $f(x, \cdot)$ is locally Lipschitz, uniformly in $x\in \Omega$.

Our first main result of this section is the following which is a nonlocal analogue of the strong comparison principle established in \cite[Theorem~3.4]{DS04}.

\begin{thm}\label{strongcomp}
Let $p>2$ and $\frac{sp}{p-2} >1$. Let
$v \in W^{s,p}_{\rm loc}(\Omega)\cap L^{p-1}_{sp}(\Rn)$ be a local weak supersolution to \eqref{main-sc} and 
$u \in W^{s,p}_{\rm loc}(\Omega)\cap L^{p-1}_{sp}(\Rn)$ be a local weak subsolution to \eqref{main-sc} satisfying $u \leq v \text{ in } \Rn$.
Furthermore, also assume that $u, v$ are locally Lipschitz in $\Omega$ and $u$  is in $C^1(\Omega)$.
Then  $v-u$ can not attend its minimum in $\Omega\setminus \mathcal{Z}_u$, where
$\mathcal{Z}_u=\{x\in \Omega\; :\; \grad u(x) =0\}$, unless $v\equiv u$ in $\Rn$.
\end{thm}

Let $x_0\in \Omega\setminus \mathcal{Z}_u$. For the simplicity of notation, we assume $x_0=0$. Since $|\grad u(0)|>0$, there exist
a unit vector $\xi$ and $R\in (0, 1)$ so that $|\partial_\xi u|>\upkappa$ in $\overline{B_{2R}}(0)\Subset \Omega$, for some positive number $\upkappa$.
We claim that we can reduce $R$, if needed, and find $\theta\in (0, 1)$ so that
\begin{equation}\label{cone-est}
|u(x)-u(y)|\geq \frac{\upkappa}{4}|x-y|\quad \text{for}\; x-y\in \cK, x, y\in B_R, \quad \text{where}\quad 
\cK=\{z\in\Rn\; :\; z\neq 0, \; |\langle {z}/{|z|}, \xi\rangle|\geq \theta\}.
\end{equation}
To prove the claim, we introduce the notation $e_{x, y}=\frac{x-y}{|x-y|}$. Using the fundamental theorem of calculus we see that
$$|u(x)-u(y)|=|x-y| \left|\int_0^1 \grad u(y+t(x-y))\cdot e_{x,y} \dt \right|.$$
Without loss of generality, we may assume that $\langle e_{x, y}, \xi\rangle \geq 0$. Otherwise, we interchange $x$ and $y$.
Note that 
\begin{align*}
\left|\int_0^1 \grad u(y+t(x-y))\cdot e_{x,y} \dt \right| &\geq 
\left|\int_0^1 \grad u(y+t(x-y))\cdot \xi \dt \right|-\int_0^1 |\grad u(y+t(x-y))| |e_{x,y}-\xi| \dt 
\\
&\geq |\grad u(y)\cdot\xi|- \int_0^1|\grad u(y+t(x-y))-\grad u(y)|\dt - \sup_{B_R}|\grad u| |e_{x,y}-\xi|
\\
&\geq \upkappa- \int_0^1|\grad u(y+t(x-y))-\grad u(y)|\dt - \sup_{B_R}|\grad u| |e_{x,y}-\xi|.
\end{align*}
Setting $R$ small enough, we can have $\int_0^1|\grad u(y+t(x-y))-\grad u(y)|\dt<\upkappa/4$. Again, since $x-y\in\cK$, we have
$$
|e_{x,y}-\xi|^2\leq 2(1-\langle e_{x,y}, \xi\rangle)\leq 2(1-\theta).
$$
Therefore, choosing $\theta$ close to $1$, we also have $\sup_{B_R}|\grad u| |e_{x,y}-\xi|<\upkappa/4$. Combining these estimates we have 
\eqref{cone-est}.

Next, we collect several auxiliary lemmas that will be used in the proof of the two key estimates in the next section. The first is an elementary inequality, which can be found in \cite[Lemma~3.1]{DKP16}.

\begin{lem}\label{lem2.2}
Let $p\geq 1$ and $\varepsilon \in (0,1]$. Then
$$|a|^p \leq |b|^p +c_p \varepsilon |b|^p + (1+c_p \varepsilon) \varepsilon^{1-p} |a-b|^p, \qquad c_p := (p-1) \Gamma (\max \{1, p-2 \}),$$
holds for every $a, b \in \Rn$. Here $\Gamma$ stands for the standard Gamma function.
\end{lem}

We also need the following elementary inequality for energy estimate.
\begin{lem}\label{lem2.3}
Let $t\in (0,1]$. Then
$(1-t)^2(\frac{1}{t} - \frac{1}{2}) \geq \frac{1}{2} (\log t)^2$.
\end{lem}
\begin{proof}
Let $f(t) = \frac{1-t}{\sqrt{t}} +\log t$. Then $f(t) \rightarrow \infty$ as $t \rightarrow 0+$ and $f(1) = 0$.
Moreover, $f'(t) = -\frac{(\sqrt{t} - 1)^2}{2t\sqrt{t}} \leq 0$. Therefore,
for $t\in (0,1)$ , $\sqrt{t} \log t \geq t-1 \Rightarrow t (\log t)^2 \leq (1-t)^2 <(1-t)^2(2-t)$. Dividing both sides by $2t$, the required inequality follows.
\end{proof}

Next one is an iteration lemma from \cite[Lemma~7.1]{Giu}.
\begin{lem}\label{lem2.4}
Let $\beta>0$ and $\{ A_j \}$ be a sequence of positive numbers satisfying
$$
A_{j+1}  \leq  c_0\, b^j\, A_j^{1+ \beta}
$$
for some $c_0>0$ and $b>1$.  If $ A_0  \leq  c_0^{-\frac{1}{\beta}} \,b^{-\frac{1}{\beta^2}}$, then we have
$$
A_j  \leq  b^{-\frac{j}{\beta}} A_0 ,
$$
which in particular, gives $\lim_{j \to \infty} A_j = 0$.
\end{lem}

For the next lemma we define, for a given domain $B$, $u_B=\smallfint_B u(x) \dx$.
\begin{lem}\label{L2.5}
Consider $u\in W^{\alpha, q}(B_1)$ with $\alpha\in (0, 1)$ and $q>1$. Also, let $\alpha q<n$ and $\supp(u)\subset B_r$ for some $r\in (0, 1)$. Then there exists a constant $C$, independent of
$u$ and $r$, so that
\begin{equation*}
\norm{u}^q_{L^{q^*}(B_1)}\leq \frac{C}{(1-r)^{\frac{q}{q^*}}} \int_{B_1}\int_{B_1}\frac{|u(x)-u(y)|^q}{|x-y|^{n+\alpha q}}\dx\dy,
\end{equation*}
where $q^*=\frac{nq}{n-\alpha q}$. Furthermore, if $u\in W^{\alpha, q}(B_{r_1})$ and $\supp(u)\subset B_{r_2}$ for some $r_2\in (0, r_1)$, we have
\begin{equation*}
\left[ \fint_{B_{r_1}} |u|^{q^*} \dx\right]^{\frac{q}{q^*}}\leq C r_1^{\alpha q} \left(\frac{r_1}{r_1-r_2}\right)^{\frac{q}{q^*}}\fint_{B_{r_1}}\int\limits_{B_{r_1}}\frac{|u(x)-u(y)|^q}{|x-y|^{n+\alpha q}}\dx\dy.
\end{equation*}
\end{lem}

\begin{proof}
From \cite[Theorem 4.10]{HV13} we know that 
$$\left[\int_{B_1}|u(x)-u_{B_1}|^{q^*} \dx\right]^\frac{q}{q^*} \leq C \int_{B_1}\int_{B_1}\frac{|u(x)-u(y)|^q}{|x-y|^{n+\alpha q}}\dx\dy.$$
Since $\supp(u)\subset B_r$, we get
$$|u_{B_1}|^q |B_1\setminus B_{r}|^\frac{q}{q^*}= \left[\int_{B_1\cap B^c_r}|u(x)-u_{B_1}|^{q^*} \dx\right]^\frac{q}{q^*}\leq C \int_{B_1}\int_{B_1}\frac{|u(x)-u(y)|^q}{|x-y|^{n+\alpha q}}\dx\dy.$$
Therefore, 
\begin{align*}
\left[\int_{B_1}|u(x)|^{q^*} \dx\right]^\frac{q}{q^*} & \leq 2^{q-1} \left[\int_{B_1}|u(x)-u_{B_1}|^{q^*} \dx\right]^\frac{q}{q^*} + 2^{q-1} |B_1|^{\frac{q}{q^*}} |u_{B_1}|^q
\\
&\leq \frac{C}{(1-r)^{\frac{q}{q^*}}} \int_{B_1}\int_{B_1}\frac{|u(x)-u(y)|^q}{|x-y|^{n+\alpha q}}\dx\dy,
\end{align*}
where we use the fact $1-r^n> 1-r$ for $r\in (0, 1)$. To prove the second part, we define $w(x)=u(r_1 x)\in W^{\alpha, q}(B_1)$ with $\supp(w)\subset B_{\frac{r_2}{r_1}}$, and the inequality follows from the first part.
\end{proof}

\subsection{Two key lemmas }
The next two lemmas provide the key estimates required to prove Theorem~\ref{strongcomp}. These results are inspired by \cite[Lemmas~3.1 and 3.2]{DKP14}, which play a central role in establishing the Harnack inequality for fractional $p$-harmonic functions. However, our estimates differ substantially from those in \cite{DKP14} in several places, primarily due to the distinct nature of our current setting.

\begin{lem}\label{L2.6}
Suppose that the hypotheses of Theorem~\ref{strongcomp} hold and we fix $R$ as in \eqref{cone-est}.  Suppose that there exist $\sigma \in (0,1]$ and $k>0$ such that for $w =v-u$ 
we have
\begin{equation}\label{EL2.6A}
|B_r \cap \{ w \geq k\}| \geq \sigma |B_r|
\end{equation} 
for some $r>0$ satisfying $\max \{ r^{(1-\frac{p-2}{sp})(p-1)}, 16r \} <R$ . Then there exists a constant $\tilde c$, independent of $\sigma, \delta, r$ and $k$, satisfying
\begin{equation*}
\left| B_{6r} \cap \{w \leq 2 \delta k \} \right| \leq \dfrac{\tilde c}{\sigma \log \frac{1}{2\delta}} |B_{6r}|
\end{equation*}
for any $\delta \in (0, \frac{1}{4})$.
\end{lem}

\begin{proof}
Denote by $\nu(\dx\dy)=|x-y|^{-n-sp}\dx\dy$.
Let $\phi$ be a smooth cutoff function supported in $B_{7r}$, $\phi = 1$ in $B_{6r}$ and $|\grad \phi| \leq \frac{C}{r}$. Let $\varepsilon >0$. Let $\tilde v = v +\varepsilon$, $\tilde w = w + \varepsilon$ and 
$\eta := \frac{\phi^2}{\tilde w}$ for $\varepsilon\in (0, 1)$. Note that $\eta$ is Lipschitz in $B_R$ and $\supp(\eta)\subset B_{7r}\Subset B_R$.
By \cite[Lemma~5.1]{DNPV} and \cite[Theorem 6.62]{Leoni} we have $\eta\in \bX(B_R)$. Using $\eta$ as test function in \eqref{main-sc}, we obtain
\begin{align*}
\int_{\Rn} \int_{\Rn} \left[ J_p(u(x) - u(y)) - J_p(v(x) - v(y)) \right]\left( \frac{\phi^2(x)}{\tilde w (x)} - \frac{\phi^2(y)}{\tilde w (y)}\right)  \dnu 
&\leq 2\int_{B_{7r}} \left[ f(x, u) - f(x, v) \right]\frac{\phi^2(x)}{\tilde w (x)} \dx 
\\
&\leq  2L \int_{B_{7r}} \dfrac{\tilde w - \epsilon}{\tilde w} \phi^2(x) \dx 
\\
&\leq 2L \int_{B_{7r}} \phi^2(x) \dx 
\\  
&\leq C r^n
\end{align*}
for some constant $C=C(n, L)$, where $L$ denote the Lipschitz constant of $f(x, \cdot)$ in $[-\sup_{B_R}(|u|+|v|), \sup_{B_R}(|u|+|v|)]$, which can be chosen independent of $x$.
We split the left-hand side in three parts as follows
\begin{align*}
& \underbrace{\int_{B_{8r}} \int_{B_{8r}} \left[ J_p(u(x) - u(y)) - J_p(v(x) - v(y)) \right]\left( \frac{\phi^2(x)}{\tilde w (x)} - \frac{\phi^2(y)}{\tilde w (y)}\right) \dnu}_{I_1}
 \\
&\quad \underbrace{+ \int_{B_{8r}^c} \int_{B_{8r}} \left[ J_p(u(x) - u(y)) - J_p(v(x) - v(y)) \right]\left( \frac{\phi^2(x)}{\tilde w (x)} \right) \dnu }_{I_2}
\\
&\qquad \underbrace{- \int_{B_{8r}} \int_{B_{8r}^c} \left[ J_p(u(x) - u(y)) - J_p(v(x) - v(y)) \right]\left( \frac{\phi^2(y)}{\tilde w (y)} \right) \dnu }_{I_3}.
\end{align*}
Therefore, we have
\begin{equation}\label{EL2.6B}
I_1+I_2+I_3 \leq Cr^n.
\end{equation}
From \cite[Lemma 2.1]{D98} we have constants $C_1, C_2$ so that
for $a <b$ and $|a| + |b| >0$, we have
\begin{align}\label{EL2.6D}
C_2 (|a| + |b|)^{p-2}(b - a) \leq J_p(b) - J_p(a) \leq C_1 (|a| + |b|)^{p-2}(b - a).
\end{align}
We also need the following notations
\begin{align}
\beta(x,y) &=(|u(x)-u(y)|+|v(x)-v(y)|)^{p-2}\label{beta}
\\
\beta_l(x, y) &=\mathbbm{1}_{\cK}(x-y)|x-y|^{p-2}, \quad \beta_r(x, y)=|x-y|^{p-2},\nonumber
\end{align}
where $\cK$ is given by \eqref{cone-est}. In view of \eqref{cone-est} and the Lipschitz property of $u, v$, we have
\begin{equation}\label{EL2.6E}
\left(\frac{\upkappa}{4}\right)^{p-2}\, \beta_l(x, y)\leq \beta(x, y)\leq \kappa\beta_r(x, y)\quad \text{for}\; x, y\in B_R,
\end{equation}
and the constant $\kappa$ depends on the Lipschitz constant of $u, v$ in $B_R$.

We first compute $I_2$. Note that $w(x) \leq w(y)\Leftrightarrow v(x)-v(y)\leq u(x)-u(y)$. Thus, using \eqref{EL2.6D},
\begin{align}\label{EL2.6F}
I_2 &= \iint\limits_{ \{ (x,y): x \in B_{8r}, y\in B_{8r}^c, w(x) \leq w(y)\}} \left[ J_p(u(x) - u(y)) - J_p(v(x) - v(y)) \right]\left( \frac{\phi^2(x)}{\tilde w (x)} \right)  \dnu \nonumber
\\
&\qquad + \iint\limits_{ \{ (x,y): x \in B_{8r}, y\in B_{8r}^c, w(x) > w(y)\}} \left[ J_p(u(x) - u(y)) - J_p(v(x) - v(y)) \right]\left( \frac{\phi^2(x)}{\tilde w (x)} \right) \dnu \nonumber
 \\
&\geq - C_1 \iint\limits_{ \{ (x,y): x \in B_{7r}, y\in B_{8r}^c, w(x) > w(y)\}} \frac{w(x)-w(y)}{\tilde w (x)} \phi^2(x) \beta(x,y) \dnu \nonumber
\\
& \geq -C_1 \iint\limits_{ \{ (x,y): x \in B_{7r}, y\in B_{8r}^c, w(x) > w(y)\}} \frac{w(x)}{\tilde w (x)}   \frac{\beta(x,y)\dx\dy}{|x-y|^{n+sp}} \nonumber
\\
&\geq -C_1  \int_{B_{7r}}\int_{B_{8r}^c} \frac{\beta(x, y)}{|x-y|^{n+sp}}\dy\dx.
\end{align}
Set $\alpha=\frac{sp-(p-2)}{2}$. Since $sp-(p-2)=2-p(1-s)<2$ and we have assumed $sp>p-2$, we get $\alpha\in (0, 1)$.
Again, since $u, v$ are Lipschitz in $\overline{B_R}$, we get for $x\in B_{7r}$ that
\begin{align*}
\int_{B_{8r}^c} \frac{\beta(x, y)}{|x-y|^{n+sp}}\dy &\leq c \int_{B_{R}\cap B_{8r}^c} |x-y|^{-n-2\alpha}\dy + c \int_{B^c_R} \frac{(1+|u(y)|+|v(y)|)^{p-2}}{|x-y|^{n+sp}}\dy
\\
&\leq \kappa \int_{B_{R}\cap B_{8r}^c} |x-y|^{-n-2\alpha}\dy 
\\
&\qquad + \kappa \left[\int_{B^c_R} \frac{(1+|u(y)|+|v(y)|)^{p-1}}{|x-y|^{n+sp}}\dy\right]^{\frac{p-2}{p-1}} 
\left[\int_{B^c_R} \frac{1}{|x-y|^{n+sp}}\dy\right]^{\frac{1}{p-1}}
\\
&\leq \kappa_1 \int_{B_{R}\cap B_{8r}^c} |y|^{-n-2\alpha}\dy 
\\
&\qquad + \kappa_1 \left[\int_{B^c_R} \frac{(1+|u(y)|+|v(y)|)^{p-1}}{|y|^{n+sp}}\dy\right]^{\frac{p-2}{p-1}} 
\left[\int_{B^c_R} \frac{1}{|y|^{n+sp}}\dy\right]^{\frac{1}{p-1}}
\\
&\leq \kappa_2 r^{-2\alpha} + \kappa_2 R^{-\frac{sp}{p-1}}\leq 2 \kappa_2\, r^{-2\alpha},
\end{align*}
where the constants $\kappa, \kappa_1, \kappa_2$ depend on $\tail_{sp,p-1}(0,R; u), \tail_{sp,p-1}(0,R; v)$
and $\norm{u}_{L^\infty(B_R)}, \norm{v}_{L^\infty(B_R)}$, but not on $r$; in the third inequality we use
$$|x-y|\geq |y|-|x|\geq |y|-7r\geq |y|-\frac{7}{8}|y|\geq \frac{|y|}{8},$$
and the last inequality follows using $r^{(1-\frac{p-2}{sp})(p-1)}<R \Leftrightarrow R^{-\frac{sp}{p-1}}< r^{-2\alpha}$. Inserting the above in \eqref{EL2.6F}, it follows that
$I_2\geq -C r^{n-2\alpha}$. A similar estimate also holds for $I_3$, giving us
$$I_2 + I_3 \geq -C r^{n-2\alpha}.$$

Next we focus on $I_1$.
Let $w(y) > w(x)$. Then by monotonicity of $J_p$, $J_p(u(x)-u(y)) - J_p (v(x)-v(y)) > 0$. Using Lemma \ref{lem2.2} we get
\begin{align*}
\frac{\phi^2(x)}{\tilde w(x)} - \frac{\phi^2(y)}{\tilde w(y)} &= \frac{1}{\tilde w(y)} \left[ \phi^2(x) \times \frac{\tilde w(y)}{\tilde w (x)} - \phi^2(y) \right]\\
&\geq  \frac{\phi^2(x)}{\tilde w(y)} \left[ \frac{\tilde w(y)}{\tilde w (x)} - 1 - c_2 \epsilon \right] - \frac{(1+c_2 \epsilon)}{\epsilon} \times \frac{|\phi(x) - \phi(y)|^2}{\tilde w (y)},
\end{align*} 
leading to
\begin{align*}
&\iint\limits_{ \{ (x,y): x \in B_{8r}, y\in B_{8r}, w(x) < w(y)\}} [J_p(u(x)-u(y)) - J_p (v(x)-v(y))]\left( \frac{\phi^2(x)}{\tilde w (x)} - \frac{\phi^2(y)}{\tilde w (y)}\right) \dnu
\\
&\geq \iint\limits_{ \{ (x,y): x \in B_{8r}, y\in B_{8r}, w(x) < w(y)\}} [J_p(u(x)-u(y)) - J_p (v(x)-v(y))]\frac{\phi^2(x)}{\tilde w(y)} \left[ \frac{\tilde w(y)}{\tilde w (x)} - 1 - c_2 \epsilon \right] \dnu
\\
&\qquad - \iint\limits_{ \{ (x,y): x \in B_{8r}, y\in B_{8r}, w(x) < w(y)\}}   [J_p(u(x)-u(y)) - J_p (v(x)-v(y))]  \frac{(1+c_2 \epsilon)}{\epsilon} \times \frac{|\phi(x) - \phi(y)|^2}{\tilde w (y)} \dnu
 \\
&:= \mathscr{J}_1 - \mathscr{J}_2.
\end{align*}
Set $\epsilon = \frac{w(y) - w(x) }{\tilde w (y)} \gamma \in (0,1]$, with $\gamma \in (0,1]$ to be chosen later. Then,
\begin{align*}
\mathscr{J}_2 &\leq C_1 \iint\limits_{ \{ (x,y): x \in B_{8r}, y\in B_{8r}, w(x) < w(y)\}}   \beta(x,y) (w(y) - w(x))   \frac{(1+c_2 \epsilon)}{\frac{w(y) - w(x) }{\tilde w (y)} \gamma} \times \frac{|\phi(x) - \phi(y)|^2}{\tilde w (y)} \dnu \\
&\leq \frac{C}{\gamma} r^{-2}  \int_{B_{8r}} \int_{B_{8r}}  \frac{|x-y|^2}{|x-y|^{n+sp-(p-2)}} \dx\dy
 \\
&\leq \frac{C}{\gamma} r^{n-2\alpha},
\end{align*}
and
\begin{align*}
\mathscr{J}_1 &\geq \iint\limits_{ \{ (x,y): x \in B_{8r}, y\in B_{8r}, w(x) < w(y)\}} \beta(x,y)\frac{w(y)-w(x)}{\tilde w(y)} \phi^2(x)\left[ \frac{\tilde w(y)}{\tilde w (x)} - 1 - c_2 \frac{w(y) - w(x) }{\tilde w (y)}
 \gamma \right] \dnu 
 \\
&\geq  \iint\limits_{ \{ (x,y): x \in B_{8r}, y\in B_{8r}, w(x) < w(y)\}} \beta(x,y) \phi^2(x) \left( \frac{w(y)-w(x)}{\tilde w(y)} \right)^2 \left[ \frac{\tilde w(y)}{\tilde w (x)} - c_2 \gamma \right] \dnu \\
&\geq C \iint\limits_{ \{ (x,y): x \in B_{6r}, y\in B_{6r}, w(x) < w(y)\}} \left| \log \left( \frac{\tilde w (x) }{\tilde w (y)} \right) \right|^2 \beta(x,y) \dnu ,
\end{align*}
where the last inequality follows by choosing $\gamma = \min\{1, \frac{1}{2 c_2}\} \in (0,1)$ and using Lemma \ref{lem2.3}. Therefore, using the symmetry, we obtain
\begin{equation}\label{Fu-A}
 I_1 \geq  C \int_{B_{6r}} \int_{B_{6r}} \left| \log \left( \frac{\tilde w (x) }{\tilde w (y)} \right) \right|^2 \beta(x,y) \dnu  -C r^{n-2\alpha}.
\end{equation}
Gathering the estimates of $I_1, I_2, I_3$ in \eqref{EL2.6B} we have
\begin{equation}\label{EL2.6G}
\int_{B_{6r}} \int_{B_{6r}} \left| \log \left( \frac{\tilde w (x) }{\tilde w (y)} \right) \right|^2 \beta(x,y) d\nu \leq Cr^{n-2\alpha}.
\end{equation}
Define $\widehat{w} := \left[ \min \{ \log\frac{1}{2\delta}, \log \frac{k+\varepsilon }{\tilde w} \} \right]_+$. Since $\widehat{w}$ is a truncation of $\log (k+\varepsilon) - \log \tilde{w}$, the energy decreases, 
giving us from \eqref{EL2.6G} that
\begin{align*}
\left(\frac{\upkappa}{4}\right)^{p-2} \int_{B_{6r}} \int_{B_{6r}} \left| \widehat{w}(x) - \widehat{w}(y) \right|^2 \beta_l(x, y) \frac{\dx\dy}{|x-y|^{n+sp}} &\leq  \int_{B_{6r}} \int_{B_{6r}} \left| \widehat{w}(x) - \widehat{w}(y) \right|^2 \beta(x,y) \dnu
 \\
&\leq  \int_{B_{6r}} \int_{B_{6r}} \left| \log \left( \frac{\tilde w (x) }{\tilde w (y)} \right) \right|^2 \beta(x,y) \dnu \\
&\leq Cr^{n-2\alpha} ,
\end{align*}
where the first inequality follows from \eqref{EL2.6E}. At this point, we let $\tilde\beta_l(x, y)=\mathbbm{1}_{\cK}(x-y)|x-y|^{-n-2\alpha}$, where $\cK$ is given by \eqref{cone-est}. Note that, since $\cK$ is
symmetric, we have
$\tilde\beta_l(x, y)=\tilde\beta_l(y, x)$ and belongs to the class of kernels given by \cite[Example~1.4]{DK20}. Thus, by \cite[Theorem 1.11]{DK20} (see also \cite{CS20}), there exists a constant
$A$, independent of $u$, satisfying
\begin{equation}\label{DK}
\int_{B_{6r}} \int_{B_{6r}} \left| \widehat{w}(x) - \widehat{w}(y) \right|^2  \frac{\dx\dy}{|x-y|^{n+2\alpha}}\leq A \int_{B_{6r}} \int_{B_{6r}} \left| \widehat{w}(x) - \widehat{w}(y) \right|^2 \tilde\beta_l(x, y) \dx\dy
\end{equation}
for all $r<\frac{1}{6}$.
Since $\tilde\beta_l(x, y)=\frac{\beta_l(x,y)}{|x-y|^{n+sp}}$ for $x, y\in B_R$, we obtain
\begin{equation*}
\int_{B_{6r}} \int_{B_{6r}} \left| \widehat{w}(x) - \widehat{w}(y) \right|^2  \frac{\dx\dy}{|x-y|^{n+2\alpha}}\leq C r^{n-2\alpha},
\end{equation*}
where the constant $C$ does not depend on $r$. Thus, applying fractional Poincar\'e inequality \cite{HV13}, we get
\begin{equation}\label{EL2.6H}
\int_{B_{6r}} |\widehat{w}(x) - \widehat{w}_{B_{6r}}| \dx \leq C r^{n/2} \left( r^{2\alpha} \int_{B_{6r}} \int_{B_{6r}} \left| \widehat{w}(x) - \widehat{w}(y) \right|^2 \frac{\dx\dy}{|x-y|^{n+2\alpha}} \right)^{\frac{1}{2}}\leq C |B_{6r}|.
\end{equation}
From the definition of $\widehat{w}$ and $\tilde w$ we have 
$$\{ \widehat{w} = 0 \} = \{ \tilde w \geq k + \varepsilon \} = \{w \geq k \}.$$
Hence, from \eqref{EL2.6A}, we obtain
\begin{align*}
\log \frac{1}{2 \delta} &= \dfrac{1}{|B_{6r} \cap \{\widehat{w} = 0 \} |} \int_{B_{6r} \cap \{ \widehat{w} = 0 \}} \left( \log \frac{1}{2\delta} -\widehat{w}(x) \right) \dx
\\
&\leq \frac{6^n}{\sigma} \left[ \log \frac{1}{2\delta} - \widehat{w}_{B_{6r}} \right].
\end{align*}
Integrating both sides over $B_{6r} \cap \{ \widehat{w} = \log \frac{1}{2\delta} \}$, we have from \eqref{EL2.6H} that
\begin{align*}
| B_{6r} \cap \{ \widehat{w}= \log \frac{1}{2\delta} \} | \log \frac{1}{2\delta} &\leq \frac{6^n}{\sigma} \int_{B_{6r} \cap \{ \widehat{w} = \log \frac{1}{2\delta}\} } |\widehat{w}(x) -  \widehat{w}_{B_{6r}}|\dx
 \\
&\leq \frac{C}{\sigma} \int_{B_{6r}} |\widehat{w}(x) -  \widehat{w}_{B_{6r}}|\dx
 \\
& \leq \frac{C}{\sigma} |B_{6r}|,
\end{align*}
$$\Rightarrow  | B_{6r} \cap \{ \tilde w \leq 2\delta(k+\varepsilon) \} | \leq \frac{C}{\sigma \log \frac{1}{2\delta} } |B_{6r}|.$$
Letting $\varepsilon \rightarrow 0$ finishes the proof.
\end{proof}
It is important to point out that $\tilde{c}$ in Lemma~\ref{L2.6} depends on $\sup_{B_R}(|u|+|v|)$, Lipschitz constants of $u, v$ in $B_R$, 
$\tail_{sp,p-1}(0, R;u), \tail_{sp,p-1}(0, R; v)$ and the constant $\upkappa$ in 
\eqref{cone-est}. This observation will be useful in proving Theorem~\ref{strongcomp}.
The next lemma is the key to prove Theorem~\ref{strongcomp}.
\begin{lem}\label{L2.7}
Consider the setting of Theorem~\ref{strongcomp} and $R$ given by \eqref{cone-est}.
Suppose that there exist $\sigma \in (0,1], k\in (0, 1)$, such that for $w =v-u$ we have
\begin{equation}\label{EL2.7A}
|B_r \cap \{ w \geq k\}| \geq \sigma |B_r|
\end{equation} 
for some $r$ satisfying $\max \{ r^{(1-\frac{p-2}{sp})(p-1)}, 16r \} <R$. Then there exists a constant $\delta \in (0, \frac{1}{4})$ such that 
\begin{equation*}
\inf_{B_{4r}} w \geq \delta k.
\end{equation*}
\end{lem}

\begin{proof}
Let $\phi$ be a smooth cutoff function supported in $B_\rho$ for some $r\leq \rho \leq 6r$ and consider the test function $w_- :=(\ell - w)_+$, for some $\ell \in (\delta k, 2 \delta k)$. As argued in Lemma ~\ref{L2.6}, $w_-\phi^2$ is a valid test function.
From \eqref{main-sc} we then have
\begin{align*}
&\int_{\Rn} \int_{\Rn}[J_p(u(x) - u(y))- J_p(v(x)- v(y))](w_{-}(x)\phi^2(x) - w_-(y) \phi^2(y))\dnu
\\
&\leq 2\int_{B_\rho} (f(x, u)-f(x, v))w_-(x)\phi^2(x) \dx
\\
&\leq 2L \int_{B_\rho} |u-v|w_-(x)\phi^2(x) \dx
\\
&\leq 2L \ell^2 |B_\rho \cap \{ w < \ell \}|,
\end{align*}
where $L$ is the local Lipschitz constant of $f(x,\cdot)$, chosen uniformly with respect to $x$. As in Lemma \ref{L2.6}, we write the left-hand side as

\begin{align*}
 & \underbrace{\int_{B_{\rho}} \int_{B_{\rho}} \left[ J_p(u(x) - u(y)) - J_p(v(x) - v(y)) \right](w_-(x)\phi^2(x) - w_- (y) \phi^2(y)) \frac{\dx\dy}{|x-y|^{n+sp}}}_{I_1} 
\\
&\quad \underbrace{+ \int_{B_{\rho}} w_-(x) \phi^2(x) \dx\int_{B_{\rho}^c} \left[ J_p(u(x) - u(y)) - J_p(v(x) - v(y)) \right] \frac{\dy}{|x-y|^{n+sp}} }_{I_2}
\\
&\qquad \underbrace{-  \int_{B_{\rho}} w_-(y) \phi^2(y) \dy\int_{B_{\rho}^c} \left[ J_p(u(x) - u(y)) - J_p(v(x) - v(y)) \right] \frac{\dx}{|x-y|^{n+sp}}  }_{I_3}.
\end{align*}
Thus, we get
\begin{equation}\label{EL2.7B}
I_1+ I_2 + I_3 \leq C \ell^2 |B_\rho \cap \{ w < \ell \}|.
\end{equation}
Let us first compute $I_2$.
Using monotonicity of $J_p$ and 
$w_-(x) > 0 \Leftrightarrow w(x) < \ell \Leftrightarrow v(x) < \ell + u(x)$, we see that
\begin{align*}
\quad J_p(u(x) - u(y) ) - J_p(v(x)-v(y))
& \geq J_p(u(x) - v(y) ) - J_p(\ell +u(x) - v(y) ) 
\\
& \geq -C [|u(x) - v(y)| + |\ell + u(x) - v(y)|]^{p-2} \ell 
\\
& \geq -C (1+|v(y)|)^{p-2}\ell.
\end{align*}
Therefore,
\begin{align*}
I_2 &\geq - C\ell^2 \int_{B_\rho\cap\{w<\ell\}} \phi^2 (x) \int_{B_\rho^c} \frac{(1+|v(y)|)^{p-2}}{|x-y|^{n+sp}}\dy
\\
&\qquad \geq -C \ell^2 \int_{B_\rho\cap\{w<\ell\}} \phi^2 (x) \sup_{x \in \supp \phi} \left( \int_{B_\rho^c} \frac{(1+|v(y)|)^{p-2}}{|x-y|^{n+sp}}\dy\right).
\end{align*}
We can estimate $I_3$ in a similar fashion, giving us 
$$I_2 + I_3 \geq -C \ell^2 \int_{B_\rho\cap\{w<\ell\}} \phi^2 (x) \sup_{x \in \supp \phi} \left( \int_{B_\rho^c} \frac{(1+|u(y)|+|v(y)|)^{p-2}}{|x-y|^{n+sp}}\dy \right)\dx.$$ 

Next we estimate $I_1$. We claim that for a given $\varepsilon>0$, there exists $C_{\varepsilon}$ such that
\begin{equation}\label{EL2.7C}
\begin{split}
&J(u(x) - u(y))-J(v(x) - v(y)) (w_-(x) \phi^2(x) - w_-(y)\phi^2(y) )
\\
& \qquad\geq \frac{C_2}{2}\frac{\upkappa^{p-2}}{4^{p-2}} \beta_l(x,y)|w_-(x) - w_-(y)|^2\max \{\phi^2(x) , \phi^2(y) \} 
\\
&\qquad \qquad -\beta_r(x,y) \varepsilon |w_-(x) - w_-(y)|^2\max \{\phi^2(x) , \phi^2(y) \} 
\\
&\qquad \qquad \quad - C_{\varepsilon} \beta_r (x,y) |\phi(x) - \phi(y)|^2 \max \{ w_-^2(x), w_-^2(y) \},
\end{split}
\end{equation}
where $C_2$ and $\upkappa$ are given by \eqref{EL2.6D} and \eqref{cone-est} respectively, $\beta_l$ and $\beta_r$ are given by \eqref{EL2.6E}.
To prove the inequality, by symmetry we may assume that $w_-(x) \geq w_-(y)$. Also, assume that $\beta(x,y) > 0$, otherwise the inequality becomes trivial.

Now we consider two situations.

\noindent\textbf{Case 1:} Let $w_-(x) \phi^2(x) - w_-(y)\phi^2(y) \geq 0$.
If $w_-(y)=0=w_-(x)$, then both sides of \eqref{EL2.7C} become 0. If $w_-(y)=0<w_-(x) \Rightarrow w(x) < \ell \leq w(y) \Rightarrow u(x) - u(y) > v(x) - v(y)$. Thus, from \eqref{EL2.6D}, we get
\begin{align}\label{EL2.7D}
&[J(u(x)-u(y))-J(v(x) - v(y))]w_-(x)\phi^2(x) \nonumber
\\
&\geq C_2 \beta(x,y) \left( w(y) - w(x) \right)w_-(x) \phi^2(x) \nonumber
\\
&\geq C_2 \beta(x,y) |w_-(x) -w_-(y)| w_-(x) \phi^2(x) \nonumber
\\
&= C_2 \beta(x,y) (w_-(x)- w_-(y))(w_-(x) \phi^2(x) -w_-(y)\phi^2(y)).
\end{align}
Next, suppose $0 < w_-(y) \leq w_-(x) \Rightarrow w(x) \leq w(y) \Rightarrow v(x) - v(y) \leq u(x) - u(y)$.
If $v(x)- v(y) = u(x) - u(y)$ the claim holds trivially. When $v(x) - v(y) < u(x) - u(y)$
\begin{align}\label{EL2.7E}
&[J(u(x) - u(y) - J(v(x) - v(y))](w_-(x) \phi^2(x) -w_-(y)\phi^2(y)) \nonumber
\\
&\geq C_2 \beta(x,y) (w(y)-w(x))(w_-(x) \phi^2(x) -w_-(y)\phi^2(y)) \nonumber
\\
&= C_2 \beta(x,y)  (w_-(y)-w_-(x))(w_-(x) \phi^2(x) -w_-(y)\phi^2(y)).
\end{align}
Note that both \eqref{EL2.7D} and \eqref{EL2.7E} are the same. Now, if  $\phi(x) \geq \phi(y)$, then
$$(w_-(x) \phi^2(x) -w_-(y)\phi^2(y)) \geq \phi^2(x) (w_-(x)-w_-(y)) = \max \{ \phi^2(x), \phi^2(y) \} (w_-(x) - w_-(y))$$
and \eqref{EL2.7C} follows.
Next, we suppose that $\phi(y) >\phi(x)$. If $w_-(x) = w_-(y)$, then \eqref{EL2.7C} follows from \eqref{EL2.7D} and \eqref{EL2.7E}. Thus we assume $w_-(x) > w_-(y)$. From Lemma~\ref{lem2.2} we have
$$\phi^2(x) \geq (1-c_2 \varepsilon)\phi^2(y) - (1+c_2\varepsilon)\frac{|\phi(x) - \phi(y)|^2}{\varepsilon}.$$
We choose  $\varepsilon = \frac{1}{\max \{1, 2c_2\}} \frac{w_-(x) - w_-(y) }{w_-(x) } \in (0,1]$. Then
\begin{align*}
&(w_-(x)- w_-(y)) (w_-(x) \phi^2(x) - w_-(y) \phi^2(y))
\\
&\geq (w_-(x)- w_-(y))[w_-(x)(1-c_2\varepsilon)\phi^2(y) - \frac{w_-(x)}{\varepsilon}(1+c_2 \varepsilon)|\phi(x) - \phi(y)|^2-w_-(y)\phi^2(y)]
\\
&\geq (w_-(x) -w_-(y))^2\phi^2(y)- \frac{1}{2}(w_-(x)- w_-(y))^2\phi^2(y) - (w_-(x))^2(1+c_2)|\phi(x)-\phi(y)|^2
\\ 
&\geq \frac{1}{2} |w_-(x) - w_-(y)|^2 \max \{\phi^2(x), \phi^2(y\} -(1+c_2)|\phi(x)-\phi(y)|^2 \max \{ w_-^2(x), w_-^2(y) \}.
\end{align*}
Thus the claim follows from \eqref{EL2.7D}, \eqref{EL2.7E} and \eqref{cone-est}.

\noindent\textbf{Case 2:} Let $w_-(x) \phi^2(x) - w_-(y) \phi^2(y) <0$. This implies $0< w_-(y) \leq w_-(x)$, and so, 
$u(x) - u(y)\geq v(x)- v(y)$. We must also have $\phi(y) >\phi(x)$. From \eqref{EL2.6D} we get
$$C_2 \beta(x,y)(w_-(x) - w_-(y))\leq J(u(x) - u(y))-J(v(x) - v(y)) \leq C_1 \beta(x,y)(w_-(x) - w_-(y)).$$
Note that for $u(x) - u(y)= v(x)- v(y)\Leftrightarrow w_-(x)=w_-(y)$, these inequalities are trivial, and in other case \eqref{EL2.6D} applies.
We write
$$w_-(x)\phi^2(x) - w_-(y)\phi^2(y) = w_-(x) (\phi(x) - \phi(y))^2 + 2 w_-(x) \phi(y)(\phi(x) - \phi(y)) + (w_-(x) - w_-(y)) \phi^2(y).
 $$
Using Young's inequality and \eqref{EL2.6E} we compute
\begin{align*}
&2 [J(u(x) - u(y)) - J(v(x)- v(y))]  w_-(x)\phi(y)(\phi(x)-\phi(y))
\\
&\quad \leq2C_1\beta(x, y) (w_-(x)-w_-(y)) w_-(x)\phi(y)|\phi(x)-\phi(y)|
\\
&\quad \leq \varepsilon \beta_r(x,y) |w_-(x) - w_-(y)|^2 \phi^2(y) + C_\varepsilon\beta_r w_-^2 (x) |\phi(x) - \phi(y)|^2.
\end{align*}
Therefore,
\begin{align*}
&J(u(x) - u(y))-J(v(x) - v(y)) (w_-(x) \phi^2(x) - w_-(y)\phi^2(y) )
\\
& \geq C_2 \beta(x,y)(w_-(x) - w_-(y))^2\max \{\phi^2(x), \phi^2(y)\}
\\
&\qquad  +2[J(u(x) - u(y)) - J(v(x)- v(y))] w_-(x)\phi(y)(\phi(x)-\phi(y))
\\
&\geq C_2 \frac{\upkappa^{p-2}}{4^{p-2}} \beta_l(x,y)(w_-(x) - w_-(y))^2\max \{\phi^2(x), \phi^2(y)\}
\\
&\qquad -\varepsilon \beta_r(x,y) |w_-(x) - w_-(y)|^2 \max\{\phi^2(x), \phi^2(y)\} - C_\varepsilon\beta_r |\phi(x) - \phi(y)|^2\max \{ w_-^2(x), w_-^2(y) \}.
\end{align*}
Hence we have proved the claim \eqref{EL2.7C}.

Using
$$|w_-(x)\phi(x) - w_-(y) \phi(y)|^2 \leq 2|w_-(x) - w_-(y)|^2 \max \{\phi^2(x), \phi^2(y) \} + 2 |\phi(x) - \phi(y)|^2 \max \{ w_-^2(x), w_-^2(y) \} $$
and 
$$|w_-(x) - w_-(y)|^2 \max \{ \phi^2(x), \phi^2(y) \} \leq 2|w_-(x)\phi(x) - w_-(y)\phi(y)|^2 + 2 \max \{ w_-^2(x), w_-^2(y)\} |\phi(x) - \phi(y)|^2,$$
and \eqref{EL2.7C}, we can estimate, with $\eta=\frac{C_2}{4}\frac{\upkappa^{p-2}}{4^{p-2}}$,
\begin{align*}
I_1 &\geq \eta \int_{B_{\rho}} \int_{B_{\rho}} |w_-(x) \phi(x) -w_-(y)\phi(y)|^2\frac{\beta_l(x,y)\dx\dy}{|x-y|^{n+sp}}
\\
&\quad -2\varepsilon \int_{B_{\rho}}\int_{B_{\rho}} |w_-(x) \phi(x) -w_-(y)\phi(y)|^2\frac{\beta_r(x,y)\dx\dy}{|x-y|^{n+sp}}
\\
&\qquad -(C_\varepsilon+\eta+2\varepsilon) \int_{B_{\rho}}\int_{B_{\rho}} \max \{ w_-^2(x), w_-^2(y) \} |\phi(x) - \phi(y) |^2 
 \frac{\beta_r(x,y)\dx\dy}{|x-y|^{n+sp}}
\\
&\geq \eta \int_{B_{\rho}} \int_{B_{\rho}} |w_-(x) \phi(x) -w_-(y)\phi(y)|^2{\tilde\beta_l(x,y)\dx\dy}
\\
&\quad -2\varepsilon \int_{B_{\rho}}\int_{B_{\rho}} |w_-(x) \phi(x) -w_-(y)\phi(y)|^2\frac{\dx\dy}{|x-y|^{n+2\alpha}}
\\
&\qquad -C \int_{B_{\rho}}\int_{B_{\rho}} \max \{ w_-^2(x), w_-^2(y) \} |\phi(x) - \phi(y) |^2 
 \frac{\dx\dy}{|x-y|^{n+2\alpha}}
\\
&\geq \frac{\eta}{A} \int_{B_{\rho}} \int_{B_{\rho}} |w_-(x) \phi(x) -w_-(y)\phi(y)|^2\frac{\dx\dy}{|x-y|^{n+2\alpha}}
\\
&\quad -2\varepsilon \int_{B_{\rho}}\int_{B_{\rho}} |w_-(x) \phi(x) -w_-(y)\phi(y)|^2\frac{\dx\dy}{|x-y|^{n+2\alpha}}
\\
&\qquad -C \int_{B_{\rho}}\int_{B_{\rho}} \max \{ w_-^2(x), w_-^2(y) \} |\phi(x) - \phi(y) |^2 
 \frac{\dx\dy}{|x-y|^{n+2\alpha}},
\end{align*} 
where in the last inequality we use \eqref{DK} and $\tilde\beta$ is same as in Lemma~\ref{L2.6}. Now we set $\varepsilon<\frac{\eta}{4A}$
and gather the estimates in \eqref{EL2.7B} we obtain
\begin{align}\label{EL2.7F}
&\int_{B_{\rho}}\int_{B_{\rho}} |w_-(x) \phi(x) - w_-(y)\phi(y)|^2 \frac{\dx\dy}{|x-y|^{n+2\alpha}}\nonumber
\\
&\quad \leq C\ell^2 \int_{B_{\rho}}\int_{B_{\rho}\cap \{w<\ell\}} |\phi(x) - \phi(y)|^2 \frac{\dx\dy}{|x-y|^{n+2\alpha}} \nonumber
\\
& \qquad +  C\ell^2 \int_{B_\rho\cap\{w<\ell\}} \phi^2 (x) \sup_{x \in \supp \phi} \left( \int_{B_\rho^c} \frac{(1+|u(y)|+|v(y)|)^{p-2}}{|x-y|^{n+sp}}\dy \right) \nonumber
\\
&\qquad +C \ell^2 |B_\rho \cap \{ w < \ell \}|.
\end{align}
Let us now define the following quantities:
$$\ell \equiv \ell_j := \delta k +\frac{\delta k}{2^{j+1}}, \qquad w_- \equiv w_j := (\ell_j - w)_+$$
$$\rho \equiv \rho_j:= 4r +\frac{2r}{2^j} \quad \text{and} \quad \tilde \rho_j = \frac{\rho_{j+1}+ \rho_j}{2} $$
for $j= 0,1,2,\ldots$. By our choice $\tilde\rho_j-\rho_{j+1}=\rho_j-\tilde \rho_j=\frac{\rho_{j}-\rho_{j+1}}{2}= r 2^{-j-1}$ and $\ell_j< 2\delta k$. Moreover,
$$w_j \geq (\ell_j - \ell_{j+1}) \mathbbm{1}_{\{w < \ell_{j+1}} \} = \frac{\delta k }{2^{2+j}} \mathbbm{1}_{\{w < \ell_{j+1}} \}  \geq 2^{-j-3} \ell_{j} \mathbbm{1}_{\{w < \ell_{j+1}} \}.$$
Let $\phi_j \in C_0^\infty(B_{\tilde \rho_j})$ be such that $0 \leq \phi_j \leq 1$, $\phi_j = 1$
 in $B_{\rho_{j+1}}$ and $|\grad \phi_j| \leq C\frac{2^{j+1}}{r}$,
for some $C$ independent of $r$ and $j$. Using Lemma~\ref{L2.5}, since $2\alpha=sp-p+2<2\leq n$, we get
\begin{align}\label{EL2.7G}
(2^{-j-3} \ell_j)^2 \left( \dfrac{|B_{\rho_{j+1}} \cap \{w < \ell_{j+1} \} |}{|B_{\rho_{j+1}}|} \right)^{\frac{2}{2^*}}
&\leq 
(\ell_j-\ell_{j+1})^2 \left( \dfrac{|B_{\rho_{j+1}} \cap \{w < \ell_{j+1} \} |}{|B_{\rho_{j+1}}|} \right)^{\frac{2}{2^*}} \nonumber
\\
&\leq 
 \left( \fint_{B_{\rho_{j+1}}}       w_j^{2^*} dx \right)^{\frac{2}{2^*}}\nonumber
 \\
 &\leq C \left( \fint_{B_{\rho_{j}}}       w_j^{2^*} \phi_j^{2^*}  dx \right)^{\frac{2}{2^*}}\nonumber
 \\
 &\leq C(\rho_j)^{2\alpha} (2^{j+3})^{\frac{2}{2^*}} \fint_{B_{\rho_{j}}}   \int\limits_{B_{\rho_{j}}} \frac{( w_j(x) \phi_j(x) -  w_j(y) \phi_j(y) )^2}{|x-y|^{n+2\alpha}} \dy\dx.
 \end{align}
 Again,
 \begin{align*}
 \fint_{B_{\rho_j}}\int\limits_{B_{\rho_j}\cap \{w<\ell_j\}} |\phi_j(x) - \phi_j(y)|^2 \frac{\dx\dy}{|x-y|^{n+2\alpha}}\dx\dy
& \leq \fint_{B_{\rho_j}}\int\limits_{B_{\rho_j}\cap \{w<\ell_j\}} \norm{\grad\phi_j}^2_\infty \frac{\dx\dy}{|x-y|^{n+2\alpha-2}}\dx\dy
 \\
 &\leq C \frac{|B_{\rho_j}\cap\{w<\ell_j\}|}{|B_{\rho_i}|} \left[\frac{2^{j+1}}{r}\right]^2 (\rho_j)^{2-2\alpha},
 \end{align*}
 and using the boundedness of $u, v$ in $B_R$, we obtain
 \begin{align*}
&\frac{1}{|B_{\rho_j}|} \int_{B_{\rho_j}\cap\{w<\ell_j\}} \phi_j^2 (x) \sup_{x \in \supp \phi_j} \left( \int_{B_{\rho_j}^c} \frac{(1+|u(y)|+|v(y)|)^{p-2}}{|x-y|^{n+sp}}\dy \right)
\\
&\quad  \leq \frac{C}{|B_{\rho_j}|} \int_{B_{\rho_j}\cap\{w<\ell_j\}} \phi_j^2 (x) \sup_{x \in \supp \phi_j} 
\left( (\rho_j-\tilde\rho_j)^{-sp} + \int_{B_{R}^c} \frac{(1+|u(y)|+|v(y)|)^{p-2}}{|x-y|^{n+sp}}\dy \right)
\\
&\quad\leq C \frac{|B_{\rho_j}\cap\{w<\ell_j\}|}{|B_{\rho_i}|}  (r 2^{-j-1})^{-sp},
 \end{align*}
 where the constant $C$ depends on the supremum of $u, v$ in $B_R$ and $\tail_{sp,p-1}(0, R;u), \tail_{sp,p-1}(0, R;v)$, and
 the last inequality is computed as in Lemma~\ref{L2.6} (see below \eqref{EL2.6F}).
Gathering the above estimates with \eqref{EL2.7G} in \eqref{EL2.7F}, and defining
$$A_j = \frac{|B_{\rho_j} \cap \{ w < \ell_j \}|}{|B_{\rho_j|}}, $$
 we arrive at 
\begin{align*}
A_{j+1}^{\frac{2}{2^*}}
\leq C (2^{2\frac{2^*+1}{2^*}})^{j+3} \left( 2^{2(j+1)} r^{-2\alpha} + 2^{(j+1)sp}r^{2-p} + 1\right) A_j
\Rightarrow A_{j+1}\leq c_0 b^j A_j^{\beta+1}
\end{align*}
for some constant $c_0>1$, where $\beta=\frac{2^*-2}{2}>0$ and $b=2^{1+2\times 2^*+sp\frac{2^*}{2}}$. The reader might have noticed that unlike \cite{DKP14}, the constant $c_0$ above depends on $r$.  Now using \eqref{EL2.7A} and Lemma~\ref{L2.6} we can find $\delta\in (0, \frac{1}{4})$
small enough to satisfy $ A_0  \leq  c_0^{-\frac{1}{\beta}} \,b^{-\frac{1}{\beta^2}}$. Thus, by Lemma~\ref{lem2.4}, we have $\lim_{j\to\infty}A_j=0$,
giving us $|B_{4r}\cap \{w<\delta k\}|=0$. This gives us the conclusion of our lemma.
\end{proof}

Now we are ready to prove Theorem~\ref{strongcomp}.
\begin{proof}[Proof of Theorem~\ref{strongcomp}]
Let $x_0\in\Omega\setminus\mathcal{Z}_u$ be such that $w(x_0)=v(x_0)-u(x_0)=0$. Choose $R\in (0, 1)$ small enough so that \eqref{cone-est} holds.
Let $r>0$ be small enough such that $\max \{ r^{(1-\frac{p-2}{sp})(p-1)}, 16r \} <R$. Define
$\upgamma := \smallfint_{B_r(x_0)}w \dx$. We can let $r$ be small enough so that $\upgamma<1$.
We claim that $\upgamma=0$. If not, then we must have 
$\upgamma >0$. Let $k =\frac{\upgamma}{2}$. Note that, since $\upgamma>0$,
\begin{equation}\label{ET2.1A}
|B_r(x_0) \cap \{ w \geq k\}| \geq \frac{\upgamma}{4M} |B_r(x_0)|,
\end{equation}
where $M := \sup_{B_r(x_0)} w>0$. In fact, if the above inequality fails to hold then we would have
\begin{align*}
\upgamma &= \fint_{B_r(x_0)}w(x) \dx
\\
&= \dfrac{1}{|B_r(x_0)|}\int\limits_{B_r(x_0)\cap \{w <k\}}w(x) \dx +\dfrac{1}{|B_r(x_0)|}\int\limits_{B_r(x_0)\cap \{w \geq k \}}w(x) \dx
\\
&\leq \frac{\upgamma}{2} + \frac{\upgamma}{4M} M<\upgamma,
\end{align*}
which can not happen. Note that \eqref{ET2.1A} is same as \eqref{EL2.7A} with $\sigma=\frac{\upgamma}{4M}\in (0, 1)$. Hence by Lemma~\ref{L2.7}
there exists a $\delta\in (0, \frac{1}{4})$ satisfying
$$0<\delta\frac{\upgamma}{2}=\delta k\leq \inf_{B_{4r}(x_0)}w=0.$$
This is a contradiction, and therefore, we must have $\upgamma=0$. This, in particular, implies that $w\equiv 0$ in $B_r(x_0)$.

Now, let $\varphi\neq 0$ be a smooth nonnegative function supported in $B_{r}(x_0)$. Using $\varphi$ as a test function, we get
\begin{align*}
\int_{\Rn} \int_{\Rn}[J_p(u(x) - u(y))- J_p(v(x)- v(y))](\varphi(x) - \varphi(y))\frac{\dx\dy}{|x-y|^{n+sp}} &\leq \int_{\Rn} (f(x, u) -f(x, v))\varphi(x) \dx
\\
&=0,
\end{align*}
since $u=v$ in $B_r(x_0)$.
Using monotonicity of $J_p$ and the fact $v\geq u$, we compute the left-hand side as follows
\begin{align*}
&\int_{\Rn} \int_{\Rn}[J_p(u(x) - u(y))- J_p(v(x)- v(y))](\varphi(x) -  \varphi(y))\frac{\dx\dy}{|x-y|^{n+sp}}
\\
&= \int_{B_r(x_0)} \int_{B_r(x_0)}[J_p(u(x) - u(y))- J_p(v(x)- v(y))](\varphi(x) - \varphi(y))\frac{\dx\dy}{|x-y|^{n+sp}} 
\\ &\quad + \int_{B^c_r(x_0)} \int_{B_r(x_0)}[J_p(u(x) - u(y))- J_p(v(x)- v(y))]\varphi(x) \frac{\dx\dy}{|x-y|^{n+sp}} 
\\ &\qquad - \int_{B_r(x_0)} \int_{B^c_r(x_0)}[J_p(u(x) - u(y))- J_p(v(x)- v(y))]\varphi(y) \frac{\dx\dy}{|x-y|^{n+sp}} 
\\
&=  \int_{B^c_r(x_0)} \int_{B_r(x_0)}[J_p(u(x) - u(y))- J_p(u(x)- v(y))]\varphi(x) \frac{\dx\dy}{|x-y|^{n+sp}} 
\\ &\quad - \int_{B_r(x_0)} \int_{B^c_r(x_0)}[J_p(u(x) - u(y))- J_p(v(x)- u(y))]\varphi(y) \frac{\dx\dy}{|x-y|^{n+sp}} 
\\
&\geq 0.
\end{align*} 
Thus, we must have 
$$ \int_{B^c_r(x_0)} \int_{B_r(x_0)}[J_p(u(x) - u(y))- J_p(u(x)- v(y))]\varphi(x) \frac{dxdy}{|x-y|^{n+sp}} =0.$$
Since $J_p$ is strictly increasing, we get $u=v$ in $B^c_r(x_0)$.
Therefore, we must have $u=v$ in $\Rn$, which completes the proof.
\end{proof}

With the help of Theorem~\ref{strongcomp} we obtain the following strong comparison principle.
\begin{thm}\label{T2.8}
Consider $p>2$.
 Let $v\in W^{s, p}_{\rm loc}(\Omega)\cap L^{p-1}_{sp}(\Rn)$ be a local weak supersolution and $u\in W^{s, p}_{\rm loc}(\Omega)\cap L^{p-1}_{sp}(\Rn)$
be a local weak subsolution to \eqref{main-sc}. Also, assume that $v\geq u$ in $\Rn$. Then $v(x_0)=u(x_0)$ for some $x_0\in\Omega$ would imply
$v\equiv u$ in $\Rn$ if one of the following holds.
\begin{itemize}
\item[(a)] $\frac{sp}{p-1}< 1$, both $u$ and $v$ are locally $\frac{sp}{p-1}+\epsilon$-H\"{o}lder in $\Omega$ for some $\epsilon>0$.
\item[(b)] $\frac{sp}{p-2}>1$, both $u$ and $v$ are locally Lipschitz in $\Omega$ with $u\in C^1(\Omega)$. Moreover, $\mathcal{Z}_u\Subset \Omega$.
\end{itemize}
If $\Omega$ is bounded, and $u, v$ are continuous in a neighbourhood of $\Omega$. Then $\min_{\partial\Omega}(v-u)>0$ implies
that $\min_{\bar\Omega} (v-u)>0$.
\end{thm}
We point out that the last part of the above result is the nonlocal analogue of \cite[Prop.~2.1]{GV89} and \cite[Theorem~1.3]{RS07}. However, unlike in their work, we do not require both the subsolution and supersolution to be nonnegative.

\begin{proof}
First we consider (a).
From Lemma~\ref{PL2.5}, we know that $(-\Delta_p)^s u$ and $(-\Delta_p)^sv$ are defined at every point in $\Omega$. Moreover,
$$(-\Delta_p)^s u\leq f(x,u) \quad \text{and}\quad (-\Delta_p)^s v\geq f(x,v)$$
are satisfied in the pointwise sense. Hence  
\begin{align*}
0=f(x_0,u(x_0))-f(x_0, v(x_0))&\geq (-\Delta_p)^s u(x_0)-(-\Delta_p)^s v(x_0)
\\
&=\int_{\Rn}[J_p(u(x_0)-u(y))-J_p(v(x_0)-v(y))]\frac{\dy}{|x_0-y|^{n+sp}}.
\end{align*}
Since $J_p(u(x_0)-u(y))-J_p(v(x_0)-v(y))\geq 0$, it follows from the above that $u\equiv v$ in $\Rn$.

To prove (b), we first show that
 $\Argmin_{\Omega}(v-u)\nsubseteq \mathcal{Z}_u$. Once this is established the result follows from Theorem~\ref{strongcomp}.
 We suppose, to the contrary, that $\Argmin_{\Omega}(v-u)\subseteq \mathcal{Z}_u$.
 Consider open bounded sets $\Omega_i, i=1,2,3,$ so that 
 $\mathcal{Z}_u\Subset\Omega_1\Subset\Omega_2\Subset\Omega_3\Subset \Omega$. Since  $\Argmin_{\Omega}(v-u)\subset \mathcal{Z}_u$, we must have
$\min_{\partial \Omega_1}(v-u)>0$. For $k\in\mathbb{N}$, define
$$u_k=\max\{\min\{u, k\}, -k\}\quad \text{and}\quad v_k=\max\{\min\{v, k\}, -k\}.$$
Set $k_0$ large enough so that for all $k\geq k_0$ we have $u_k=u$ and $v_k=v$ in $\Omega_3$. Also, define
\begin{align*}
f_k(x)&=f(x,u(x))+\int_{\Rn \setminus \Omega_3} \left[J_p(u(x)-u_k(y)) - J_p(u(x)-u(y))\right] \frac{\dy}{|x-y|^{n+sp}},
\\
g_k(x)&=f(x,v(x))+\int_{\Rn \setminus \Omega_3} \left[J_p(v(x)-v_k(y)) - J_p(v(x)-v(y))\right]\frac{\dy}{|x-y|^{n+sp}}.
\end{align*}
For $k\geq k_0$, the above functions are well-defined and also continuous in $\bar\Omega_2$. By the dominated convergence theorem, 
$f_k \rightarrow f(\cdot, u(\cdot))$ and $g_k \rightarrow g(\cdot, v(\cdot))$ locally uniformly in $\Omega_3$, as $k\to\infty$. Furthermore, since $u$ is a viscosity subsolution and
$v$ is a viscosity supersolution (which follows from the equivalence in \cite{KKL19}), we get
$$(-\Delta_p)^s u_k\leq f_k\quad \text{and}\quad (-\Delta_p)^sv_k\geq g_k\quad \text{in}\; \Omega_2$$
for all $k\geq k_0$. Also, $u_k, v_k\in L^\infty(\Rn)$.
Consider two sub-sequences $u_{k,n}\in L^\infty(\Rn)$ and $v_{k,n}\in L^\infty(\Rn)$, upper and lower semi continuous, respectively, such that
\begin{itemize}
\item $u_{k,m} = u_k$ in $\Omega_3$ for all $m\in\mathbb{N}$,
\item $v_{k,m} = v_k$ in $\Omega_3$ for all $m\in\mathbb{N}$,
\item $u_{k,m} \rightarrow u$ a.e. in $\Rn \setminus \Omega_3$,
\item $v_{k,m} \rightarrow v$ a.e. in $\Rn \setminus \Omega_3$,
\item $\norm{u_{k,m}}_{L^\infty}\leq \norm{u_{k}}_{L^\infty}+1$
and $\norm{v_{k,m}}_{L^\infty}\leq \norm{v_{k}}_{L^\infty}+1$
for all $m\geq 1$.
\end{itemize}
Define
\begin{align*}
f_{k,m}(x)&=f_k(x) +\int_{\Rn \setminus \Omega_3} \left[J_p(u_k(x)-u_{k,m}(y)) - J_p(u_k(x)-u_{k}(y))\right] \frac{\dy}{|x-y|^{n+sp}},
\\
g_{k,m}(x)&=g_k(x)+\int_{\Rn \setminus \Omega_3} \left[J_p(v_k(x)-v_{k,m}(y)) - J_p(v_k(x)-v_{k}(y))\right] \frac{\dy}{|x-y|^{n+sp}}.
\end{align*}
The above functions are well-defined and also continuous in $\bar\Omega_2$.
By the dominated convergence theorem, $f_{k,m} \rightarrow f_k$ and $g_{k,m} \rightarrow g_k$ as $m\to\infty$, locally uniformly in $\Omega_3$. It is easily seen that
$$(-\Delta_p)^s u_{k,m}\leq f_{k,m}\quad \text{and}\quad (-\Delta_p)^s v_{k,m}\geq g_{k,m}\quad \text{in}\; \Omega_2.
$$
Let $(u_{k,m})^\varepsilon$, $(v_{k,m})_\varepsilon$ denote the sup and inf-convolution to
$u_{k,m}$ and $v_{k,m}$, respectively. Furthermore, by Lemma~\ref{PL2.2}, we have
$$(-\Delta_p)^s (u_{k,m})^\varepsilon \leq f_{k,m} + \eta_\varepsilon \quad \text{and}\quad (-\Delta_p)^s(v_{k,m})_\varepsilon\geq g_{k,m}+\tilde\eta_\varepsilon\quad \text{in}\; \Omega_1,$$
where $\eta_\varepsilon, \tilde\eta_\varepsilon\to 0$ as $\varepsilon\to 0$, uniformly over $\bar\Omega_1$. Since 
$(v_{k,m})_\varepsilon-(u_{k,m})^\varepsilon\to v_{k,m}-u_{k,m}$ as $\varepsilon\to 0$, uniformly over $\bar\Omega_1$, we have
$$
\min_{\bar\Omega_1}((v_{k,m})_\varepsilon-(u_{k,m})^\varepsilon)\to \min_{\bar\Omega_1}(v_{k,m}-u_{k,m}).
$$
Again, letting $\mathcal{M}=\Argmin_{\bar\Omega_1}(v-u)={\Argmin}_{\bar\Omega_1}(v_k-u_k) =\Argmin_{\bar\Omega_1}( v_{k,m} - u_{k,m})$ for $k\geq k_0$, we get that
$$ {\Argmin}_{\bar\Omega_1}((v_{k,n})_\varepsilon-(u_{k,n})^\varepsilon)\ni x_{k, m,\varepsilon}
\to 
\mathcal{M}\quad \text{as}\; \varepsilon\to 0.
$$
Since $\min_{\partial \Omega_1}(v-u)>0$, we have $\mathcal{M}\subseteq \mathcal{Z}_u \subset \Omega_1^\circ$, and $x_{k,m, \varepsilon}\in \Omega_1^\circ$ for all
small enough $\varepsilon$. Using Lemmas~\ref{PL2.3} and \ref{PL2.4} we obtain 
\begin{equation}\label{T3.8A}
(-\Delta_p)^s (u_{k,m})^\varepsilon(x_{k,m,\varepsilon}) \leq (f_{k,m} + \eta_\varepsilon)(x_{k,m,\varepsilon}) \quad \text{and}\quad (-\Delta_p)^s(v_{k,m})_\varepsilon(x_{k,m,\varepsilon})\geq (g_{k,m}+\tilde\eta_\varepsilon)(x_{k,m,\varepsilon}).
\end{equation}
Let $0< 2r<\dist(\mathcal{Z}_u, \partial\Omega_1) $, and by the above discussion, for all small $\varepsilon$ we would have
$B_r(x_{k,m,\varepsilon})\subset \Omega_1$. Thus, from \eqref{T3.8A}, we obtain
\begin{align*}
&f_{k,m}(x_{k,m,\varepsilon})-g_{k,m}(x_{k,m,\varepsilon}) + \eta_\varepsilon(x_{k,m,\varepsilon})-\tilde\eta_\varepsilon(x_{k,m,\varepsilon})
\\
&\geq \int_{\Rn}[J_p((u_{k,m})^\varepsilon(x_{k,m,\varepsilon})-
(u_{k,m})^\varepsilon(x_{k,m,\varepsilon}+z))-J_p((v_{k,m})_\varepsilon(x_{k,m,\varepsilon})-
(v_{k,m})_\varepsilon(x_{k,m,\varepsilon}+z))]\frac{\dz}{|z|^{n+sp}}
\\
&\geq \int_{B^c_r}[J_p((u_{k,m})^\varepsilon(x_{k,m,\varepsilon})-
(u_{k,m})^\varepsilon(x_{k,m,\varepsilon}+z))-J_p((v_{k,m})_\varepsilon(x_{k,m,\varepsilon})-
(v_k)_\varepsilon(x_{k,m,\varepsilon}+z))]\frac{\dz}{|z|^{n+sp}},
\end{align*}
where in the last line we use the fact 
$$(v_{k,m})_\varepsilon(x_{k,m,\varepsilon})-(u_{k,m})^\varepsilon(x_{k,m,\varepsilon})\leq (v_{k,m})_\varepsilon(x_{k,m,\varepsilon}+z)-(u_{k,m})^\varepsilon(x_{k,m,\varepsilon}+z)
\quad \text{for}\; z\in B_r,$$
together with the monotonicity of $J_p$. Letting 
$\varepsilon\to 0$ and extracting a subsequence of $x_{k,m,\varepsilon}$, converging to $x_{k,m}\in\mathcal{M}$, we obtain from above
\begin{align*}
f_{k,m}(x_{k,m})-f_{k,m}(x_{k})
\geq \int_{B^c_r}[J_p(u_{k,m}(x_{k,m})-u_{k,m}(x_{k,m}+z))-J_p(v_{k,m}(x_{k,m})-
v_{k,m}(x_{k,m}+z))]\frac{\dz}{|z|^{n+sp}}.
\end{align*}
Now first we let $m \rightarrow \infty$ and then let $k\to\infty$, and extract a subsequence of $\{x_{k,m}\}$ that converges to
some $\bar{x}\in\mathcal{M}$ to see that 
$$0\leq \int_{B^c_r}[J_p(u(\bar x)-u(\bar{x}+z))-J_p(v(\bar{x})-
v(\bar{x}+z))]\frac{\dz}{|z|^{n+sp}}\leq 0,$$
by the dominated convergence theorem. Here we use the fact $v(x_{k,m})=u(x_{k,m})$, implying 
$$f(x_{k,m}, u(x_{k,m}))=f(x_{k,m}, v(x_{k, m})).$$
 Hence we have $u(z)=v(z)$ in $B^c_r(\bar{x})$. Since $r$ can be chosen arbitrary small, we must have $u\equiv v$
in $\Rn$. This contradicts our hypothesis $\min_{\partial \Omega_1}(v-u)>0$. Hence $\Argmin_{\Omega}(v-u)\nsubseteq \mathcal{Z}_u$ 
and therefore, the proof of (b) follows from Theorem~\ref{strongcomp}.

Proof of last part of the theorem follows easily from the above. Note that the proof of (b) implies that if $\min_{\partial\Omega}(v-u)>0$, then we must have
$\min_{\bar\Omega}(v-u)>0$.
\end{proof}


\begin{rem}
We need $f$ to be locally Lipschitz only for the case (b) in Theorem~\ref{T2.8}. In all other cases, the proof works for continuous $f$.
\end{rem}

\begin{rem}
The regularity results of \cite{BDLMS24a, BS25} justify the assumption in Theorem~\ref{T2.8}(a). Specifically, when sub- or supersolutions arise from $(-\Delta_p)^s w = g(x)$ with $\alpha$-H\"{o}lder data $g$, \cite{BS25} ensures that $w$ is locally almost $\min\{1, \frac{sp+\alpha}{p-1}, \frac{sp}{p-2}\}$-H\"{o}lder continuous.
\end{rem}

\section{Strong comparison principle for reflected solutions}\label{S-anti}
In this section, we establish a strong comparison principle for solutions reflected with respect to a hyperplane, which serves as a key ingredient in proving the monotonicity of solutions. We consider the equation
\begin{equation}\label{main-anti}
(-\Delta_p)^s u = f(u) \quad \text{in}\; \Omega,
\end{equation}
where $f$ is locally Lipschitz.  Our approach involves extending the results of Lemmas~\ref{L2.6} and~\ref{L2.7} to the reflected setting.

To state the main result of this section we need a few notations. Consider a unit vector $e$. For $\lambda\in \R$ we define
\begin{gather*}
\Sigma_\lambda=\{x\in\Rn\; : \; x\cdot e<\lambda\}, \quad T_\lambda=\{x\in\Rn: \; x\cdot e=\lambda\}=\partial \Sigma_\lambda, 
\\ 
x^\lambda=2[\lambda-(x\cdot e)]e + x, \quad u_\lambda(x)=u(x^\lambda).
\end{gather*}
Also, for a given domain $D$ we denote by $D^\lambda=\{x^\lambda\; :\; x\in D\}$, the reflection of $D$ with respect to $T_\lambda$.

Let us begin with the following lemma.
\begin{lem}\label{L3.1}
Let $u\in W^{s, p}_{\rm loc}(\Omega)\cap L^\infty(\Rn)$ be a local weak solution to \eqref{main-anti}. Consider $\lambda\in \R$ so that $\Omega\cap\Sigma_\lambda\neq \emptyset$
and $(\Omega\cap\Sigma_\lambda)^\lambda\subset\Omega$.
Consider $x_0 \in \Omega\cap\Sigma_\lambda$ and $R_0>0$ such that $B_{R_0} (x_0) \Subset \Omega\cap\Sigma_\lambda$. Also, assume that $u_\lambda \geq u$ in $\Sigma_\lambda$, 
$u_\lambda(x_0)=u(x_0)$ and
$u_\lambda-u\not\equiv 0$. Then there exists $R\in (0, R_0)$ and $\bar{u}, \bar{u}_\lambda$ satisfying the following.
\begin{enumerate}
\item[(i)] $\bar{u}=u$ and $\bar{u}_\lambda = u_\lambda$ in $B_{R}(x_0)$.
\item[(ii)] $\bar{u}_\lambda - \bar{u}$ is antisymmetric in $\Sigma_\lambda \setminus B_{R}(x_0)$. That is, for $x\in \Sigma_\lambda \setminus B_{R}(x_0)$ we have
$(\bar{u}_\lambda - \bar{u})(x^\lambda)=-(\bar{u}_\lambda - \bar{u})(x)$. Moreover, $\bu_\lambda(x^\lambda)=\bu(x)$ for $x\in \Sigma_\lambda \setminus B_{R}(x_0)$.
\item[(iii)] $\bu_\lambda \geq \bu $ in $\Sigma_\lambda$.
\item[(iv)] $\bu_\lambda = \bu$ in $B_{R}^\lambda (x_0)$.
\item[(v)] $(-\Delta_p)^s \bu_\lambda(x) - (-\Delta_p)^s \bu(x) \geq f(\bu_{\lambda}(x)) - f(\bu(x))$ in $B_R(x_0)$ in the (local) weak sense.
\end{enumerate}
\end{lem}

\begin{proof}
For the sake of simplicity we assume that $x_0=0$. Since $u_\lambda-u\gneq$ in $\Sigma_\lambda$, we can find a compact set $K\Subset \Sigma_\lambda$ and $R\in (0, R_0)$ small enough so that
$K \cap \bar{B}_{R} (x_0) = \emptyset$ and $w=u_\lambda - u \geq \varrho>0$ in $K$ for some constant $\varrho$. Define
\begin{align*}
\tilde u_\lambda(x) &= u_\lambda(x) -\frac{\varrho}{8} \mathbbm{1}_{K}(x),
\\
\tilde u(x) &= u(x) -\frac{\varrho}{8} \mathbbm{1}_{K^\lambda}(x).
\end{align*}
Define $\bu := \tilde u$ and 
\[
\bu_\lambda:= 
\left\{\begin{array}{ll}
\tilde u & \text{in}\; B_R^\lambda,
 \\[2mm]
 \tilde u_\lambda & \text{ otherwise} .
 \end{array}
\right.
\] 
It is easy to verify that $(i)-(iv)$ are satisfied. So we only prove $(v)$. In what follows, we check $(v)$ in the pointwise sense, and since the function $u, u_\lambda$ have been modified in sets
disjoint from $\bar{B}_R$ to produce $\bu$ and $\bu_\lambda$, one can easily verify that the inequality also holds in the weak sense.

First we show that there exist constants $\kappa_1, \kappa_2$, independent of $R$, satisfying
for $x \in B_R \subset B_{R_0}$ 
\begin{equation}\label{EL3.1A}
(-\Delta_p)^s \tilde u_\lambda(x) - (-\Delta_p)^s \tilde u(x) \geq f(\tilde u_{\lambda}(x)) - f(\tilde u(x)) + \left(\kappa_1 \varrho^{p-1} - \kappa_2 \sup_{B_R} |u_\lambda - u|\right) |K|.
\end{equation}
Indeed, since $u$ and $u_\lambda$ satisfy \eqref{main-anti} in $B_R$, for $x\in B_R$ we get
\begin{align*}
&(-\Delta_p)^s \tilde u_\lambda(x) - (-\Delta_p)^s \tilde u(x) 
\\
&= [(-\Delta_p)^s \tilde u_\lambda(x)-(-\Delta_p)^s u_\lambda(x)] - [(-\Delta_p)^s \tilde u(x)-(-\Delta_p)^s u(x)] 
 + f(u_\lambda(x)) - f(u(x))
\\
&= \int_K [J_p(u_\lambda(x)-u_\lambda(y)+\frac{\varrho}{8})-J_p(u_\lambda(x)-u_\lambda(y))]\frac{\dy}{|x-y|^{n+sp}}
\\
&\quad - \int_{K^\lambda} [J_p(u(x)-u(y)+\frac{\varrho}{8})-J_p(u(x)-u(y))]\frac{\dy}{|x-y|^{n+sp}}
 + f(\bu_\lambda(x)) - f(\bu(x))
\\
&= \int_K [J_p(u_\lambda(x)-u_\lambda(y)+\frac{\varrho}{8})-J_p(u_\lambda(x)-u_\lambda(y))]\frac{\dy}{|x-y|^{n+sp}}
\\
&\quad - \int_{K} [J_p(u(x)-u_\lambda(y)+\frac{\varrho}{8})-J_p(u(x)-u_\lambda(y))]\frac{\dy}{|x-y^\lambda|^{n+sp}}
+ f(\bu_\lambda(x)) - f(\bu(x))
\\
&= \underbrace{\int_K [J_p(u_\lambda(x)-u_\lambda(y)+\frac{\varrho}{8})-J_p(u_\lambda(x)-u_\lambda(y))]\left(\frac{1}{|x-y|^{n+sp}}-\frac{1}{|x-y^\lambda|^{n+sp}}\right)\dy}_{I_1}
\\
&+ \underbrace{ \int_{K} [J_p(u_\lambda(x)-u_\lambda(y)+\frac{\varrho}{8})-J_p(u(x)-u_\lambda(y)+\frac{\varrho}{8})+J_p(u(x)-u_\lambda(y))-J_p(u_\lambda(x)-u_\lambda(y))]\frac{\dy}{|x-y^\lambda|^{n+sp}}}_{=I_2}
\\
&\qquad + f(\bu_\lambda(x)) - f(\bu(x)).
\end{align*}
Since $\dist(B_{R_0}, K^\lambda)>0$, we have 
$$
\inf_{x\in B_{R_0}}\inf_{y\in K}|x-y^\lambda|>0 \quad \text{and}\quad \inf_{x\in B_{R_0}}\inf_{y\in K}\left(\frac{1}{|x-y|^{n+sp}}-\frac{1}{|x-y^\lambda|^{n+sp}}\right)>0.
$$
Using \eqref{EL2.6D} we also have
\begin{align*}
|J_p(u_\lambda(x)-u_\lambda(y)+\frac{\varrho}{8})-J_p(u(x)-u_\lambda(y)+\frac{\varrho}{8})| & \leq C|u_\lambda(x)-u(x)|\leq C \sup_{B_R}|u_\lambda-u|
\\
|J_p(u(x)-u_\lambda(y))-J_p(u_\lambda(x)-u_\lambda(y))|& \leq C|u_\lambda(x)-u(x)|\leq C \sup_{B_R}|u_\lambda-u|,
\end{align*}
where the constant $C$ depends on the $L^\infty$ norm of $u$. Hence $I_2\geq -C |K| \sup_{B_R}|u_\lambda-u|$ for some constant $C$, independent of $R$. Again, by \eqref{EL2.6D}, we see that for
$b>a$
$$J_p(b)-J_p(a)\geq C_2 (|a|+|b|)^{p-2}(b-a)\geq C_2 (|b-a|)^{p-2}(b-a)=C_2 (b-a)^{p-1}.$$
Thus, for some constant $\kappa_1$, we would have $I_1\geq \kappa_1 \varrho^{p-1}|K|$. Combining these estimates we get \eqref{EL3.1A}.

Next we compute, using \eqref{EL2.6D}, that
\begin{align}\label{EL3.1B}
(-\Delta_p)^s \bu_\lambda(x) - (-\Delta_p)^s \tilde u_\lambda (x)&=\int_{B^\lambda_R}[J_p(u_\lambda(x)-u(y))-J_p(u_\lambda(x)-u_\lambda(y))]\frac{\dy}{|x-y|^{n+sp}}\nonumber
\\
&\geq -C_3 \sup_{B_R}|u_\lambda-u| R^n,
\end{align}
where the constant $C_3$ depends on the $L^\infty$ norm of $u$ and $R_0$, and in the last line we us the fact that
$$ \inf_{x\in B_{R_0}}\inf_{y\in B^\lambda_{R_0}}|x-y|>0.$$
Since $u_\lambda(0)-u(0)=0$, choosing $R$ small enough, depending on $\varrho$ and $|K|$, so that 
$$\left(\kappa_1 \varrho^{p-1} - \kappa_2 \sup_{B_R} |u_\lambda - u|\right) |K|-C_3 \sup_{B_R}|u_\lambda-u| R^n>0,$$
we establish $(v)$ from \eqref{EL3.1A} and \eqref{EL3.1B}. Hence the proof.
\end{proof}

With the help of Lemma~\ref{L3.1}, we can now produce analogous version of Lemmas~\ref{L2.6} and ~\ref{L2.7} for the reflected solutions. Let $u\in W^{s, p}_{\rm loc}(\Omega)\cap L^\infty(\Rn)$
be a solution to \eqref{main-anti} which is in $C^1(\Omega)$. As before, we define $\mathcal{Z}_u=\{x\in \Omega\; :\; \grad u(x)=0\}$. Let $\lambda\in \R$ be such that
$\Omega\cap\Sigma_\lambda\neq \emptyset$ and $(\Omega\cap\Sigma_\lambda)^\lambda\subset\Omega$. Now we consider a point $x_0\in \Omega\cap\Sigma_\lambda\setminus\mathcal{Z}_u$.
Again, due to simplicity, we let $x_0=0$.
It can be easily seen that the reasoning of Section~\ref{S-SCP} (see \eqref{cone-est}) goes through giving us $R>0$ small enough so that 
$B_R\Subset (\Omega\cap\Sigma_\lambda)\setminus\mathcal{Z}_u$, $|\xi|=1$, and
\begin{equation}\label{cone-anti}
|u(x)-u(y)|\geq {\upkappa}|x-y|\quad \text{for}\; x-y\in \cK, x, y\in B_R, \quad \text{where}\quad 
\cK=\{z\in\Rn\; :\; z\neq 0, \; |\langle {z}/{|z|}, \xi\rangle|\geq \theta\},
\end{equation}
for some $\upkappa>0$ and $\theta\in (0, 1)$. We can even choose $R$ small enough so that the conclusion of Lemma~\ref{L3.1} holds, provided $u_\lambda\gneq u$ in $\Sigma_\lambda$.

\begin{lem}\label{L3.2}
Let $p>2$, $\frac{sp}{p-2} >1$, and $\lambda$ be such that the conditions of Lemma~\ref{L3.1} hold. Also, assume that $u\in C^1(\Omega)$, $|\grad u(x_0)|\neq 0$ and choose $R$ so that \eqref{cone-anti} 
and Lemma~\ref{L3.1} hold.
Suppose that there exist $\sigma \in (0,1], k>0,$ such that for $w =\bu_\lambda-\bu$ we have
\begin{equation}\label{EL3.2A}
|B_r \cap \{ w \geq k\}| \geq \sigma |B_r|
\end{equation} 
for some $r>0$ satisfying $\max \{ r^{(1-\frac{p-2}{sp})(p-1)}, 16r \} <R$ . Then there exists a constant $\tilde c$, independent of $\sigma, \delta, r$ and $k$, such that
\begin{equation*}
\left| B_{6r} \cap \{w \leq 2 \delta k \} \right| \leq \dfrac{\tilde c}{\sigma \log \frac{1}{2\delta}} |B_{6r}|
\end{equation*}
for any $\delta \in (0, \frac{1}{4})$.
\end{lem}

\begin{proof}
We broadly follow the proof of Lemma~\ref{L2.6} and only detail those steps that require modification. Let $\phi$ be a smooth cutoff function with support in $B_{7r}$, $\phi \equiv 1$ in $B_{6r}$ and 
$|\grad \phi| \leq \frac{C}{r}$. Let $\varepsilon >0$. Let $\tilde \bu_\lambda = \bu_\lambda +\varepsilon$, $\tilde w = w + \varepsilon$ and define $\eta=\frac{\phi^2}{\tilde w}$
By Lemma~\ref{L3.1}(iii), $\eta$ is well-defined a valid test function. Using $\eta$ as test function in Lemma~\ref{L3.1}(v),  we get
\begin{align*}
\int_{\Rn} \int_{\Rn} \left[ J_p(\bu(x) - \bu(y)) - J_p(\bu_\lambda(x) - \bu_\lambda(y)) \right]\left( \frac{\phi^2(x)}{\tilde w (x)} - \frac{\phi^2(y)}{\tilde w (y)}\right) d \nu \leq C r^n,
\end{align*}
where $C$ depends on the local Lipschitz constant of $f$. Again,
\begin{align*}
{\rm LHS} &= \underbrace{\int_{B_{8r}} \int_{B_{8r}} \left[ J_p(\bu(x) - \bu(y)) - J_p(\bu_\lambda(x) - \bu_\lambda(y)) \right]\left( \frac{\phi^2(x)}{\tilde w (x)} - \frac{\phi^2(y)}{\tilde w (y)}\right) \dnu}_{I_1} 
\\
&\quad \underbrace{+ \int_{B_{8r}^c} \int_{B_{8r}} \left[ J_p(\bu(x) - \bu(y)) - J_p(\bu_\lambda(x) - \bu_\lambda(y)) \right]\left( \frac{\phi^2(x)}{\tilde w (x)} \right)  \dnu }_{I_2}
\\
&\qquad \underbrace{- \int_{B_{8r}} \int_{B_{8r}^c} \left[ J_p(\bu(x) - \bu(y)) - J_p(\bu_\lambda(x) - \bu_\lambda(y)) \right]\left( \frac{\phi^2(y)}{\tilde w (y)} \right) \dnu }_{I_3}.
\end{align*}
Estimate of $I_1$ is same as that in Lemma~\ref{L2.6} (see \eqref{Fu-A}), giving us
$$I_1 \geq  C \int_{B_{6r}} \int_{B_{6r}} \left| \log \left( \frac{\tilde w (x) }{\tilde w (y)} \right) \right|^2 \beta(x,y) \dnu  -C r^{n-2\alpha},$$
with $\beta(x, y)=(|\bu(x)-\bu(y)|+|\bu_\lambda(x)-\bu_\lambda(y)|)^{p-2}$ (see \eqref{beta}).
Let us now estimate $I_2$.
\begin{align*}
I_2 &= \int_{B_{8r}^c} \int_{B_{8r}} \left[ J_p(\bu(x) - \bu(y)) - J_p(\bu_\lambda(x) - \bu_\lambda(y)) \right] \frac{\phi^2(x)}{\tilde w (x)}  \frac{\dx\dy}{|x-y|^{n+sp}} 
\\
&= \int_{B_{R}^c \cap \Sigma_\lambda} \int_{B_{8r}} \left[ J_p(\bu(x) - \bu(y)) - J_p(\bu_\lambda(x) - \bu_\lambda(y)) \right] \frac{\phi^2(x)}{\tilde w (x)}  \frac{\dx\dy}{|x-y|^{n+sp}} 
\\
&\quad + \int_{\Sigma_\lambda^c \setminus B_R^\lambda} \int_{B_{8r}} \left[ J_p(\bu(x) - \bu(y)) - J_p(\bu_\lambda(x) - \bu_\lambda(y)) \right] \frac{\phi^2(x)}{\tilde w (x)}  \frac{\dx\dy}{|x-y|^{n+sp}}  
\\
&\qquad + \underbrace{\int_{ B_R^\lambda} \int_{B_{8r}} \left[ J_p(\bu(x) - \bu(y)) - J_p(\bu_\lambda(x) - \bu_\lambda(y)) \right] \frac{\phi^2(x)}{\tilde w (x)}  \frac{\dx\dy}{|x-y|^{n+sp}} 
}_{I_{2,2}} 
\\
&\quad\qquad + \underbrace{\int_{ B_R \setminus B_{8r}} \int_{B_{8r}} \left[ J_p(\bu(x) - \bu(y)) - J_p(\bu_\lambda(x) - \bu_\lambda(y)) \right] \frac{\phi^2(x)}{\tilde w (x)}  \frac{\dx\dy}{|x-y|^{n+sp}} 
}_{I_{2,3}}
\\
&:= I_{2,1} + I_{2,2} + I_{2,3},
\end{align*}
where $I_{2,1}$ represent sum of the first two terms in the second equality. Using Lemma~\ref{L3.1}(iv), we compute
\begin{align*}
I_{2,2} &\geq \int_{ B_R^\lambda} \int_{B_{8r}} \left[ J_p(\bu(x) - \bu(y)) - J_p(\bu_\lambda(x) - \bu(y)) \right]\left( \frac{\phi^2(x)}{\tilde w (x)} \right) \mathbbm{1}_{\{w(x)> w(y)\}} (x,y)
\frac{\dx\dy}{|x-y|^{n+sp}}
\\
& \geq -C \int_{ B_R^\lambda} \int_{B_{8r}} w(x) \left( \frac{\phi^2(x)}{\tilde w (x)} \right) \mathbbm{1}_{\{w(x)> w(y)\}} (x,y)\frac{\beta(x,y)\dx\dy}{|x-y|^{n+sp}}
\\
&\geq -Cr^{n-2\alpha},
\end{align*}
where last inequality is obtained by a similar calculation as we did in estimating $I_2$ in Lemma \ref{L2.6} and $2\alpha=sp-(p-2)$. Also,
\begin{align*}
I_{2,3} &\geq \int_{ B_R \setminus B_{8r}} \int_{B_{8r}} \left[ J_p(\bu(x) - \bu(y)) - J_p(\bu_\lambda(x) - \bu_\lambda(y)) \right]\left( \frac{\phi^2(x)}{\tilde w (x)} \right) \mathbbm{1}_{\{w(x)> w(y)\}} (x,y)\frac{\dx\dy}{|x-y|^{n+sp}}
\\
& \geq -C \int_{ B_R \setminus B_{8r}} \int_{B_{8r}} [w(x)-w(y)] \left( \frac{\phi^2(x)}{\tilde w (x)} \right) \mathbbm{1}_{\{w(x)> w(y)\}} (x,y)\frac{\beta(x,y)\dx\dy}{|x-y|^{n+sp}}
\\& \geq -C \int_{ B_R \setminus B_{8r}} \int_{B_{8r}} w(x) \left( \frac{\phi^2(x)}{\tilde w (x)} \right) \mathbbm{1}_{\{w(x)> w(y)\}} (x,y)\frac{\beta(x,y)\dx\dy}{|x-y|^{n+sp}}
\\
&\geq -Cr^{n-2\alpha},
\end{align*}
which is same as in Lemma~\ref{L2.6}. For $I_{2,1}$ we use the antisymmetric property  to compute
\begin{align*}
I_{2,1}&= \int_{B_{R}^c \cap \Sigma_\lambda} \int_{B_{8r}} \left[ J_p(\bu(x) - \bu(y)) - J_p(\bu_\lambda(x) - \bu_\lambda(y)) \right]\left( \frac{\phi^2(x)}{\tilde w (x)} \right) \frac{\dx\dy}{|x-y|^{n+sp}} 
\\
&\quad + \int_{B_{R}^c \cap \Sigma_\lambda} \int_{B_{8r}} \left[ J_p(\bu(x) - \bu_\lambda(y)) - J_p(\bu_\lambda(x) - \bu(y)) \right]\left( \frac{\phi^2(x)}{\tilde w (x)} \right) \frac{\dx\dy}{|x-y^\lambda|^{n+sp}} 
\\
&= \int_{B_{R}^c \cap \Sigma_\lambda} \int_{B_{8r}} \left[ \frac{1}{|x-y|^{n+sp}} -\frac{1}{|x-y^\lambda|^{n+sp}}\right] \left[ J_p(\bu(x) - \bu(y)) - J_p(\bu_\lambda(x) - \bu_\lambda(y)) \right]\left( \frac{\phi^2(x)}{\tilde w (x)} \right) \dx\dy
\\
&\qquad  + \int_{B_{R}^c \cap \Sigma_\lambda} \int_{B_{8r}} \Bigl[ J_p(\bu(x) - \bu_\lambda(y))- J_p(\bu_\lambda(x) - \bu_\lambda(y) 
 \\
 &\hspace{8em} + J_p (\bu(x) - \bu(y)) - J_p(\bu_\lambda(x) - \bu(y)) \Bigr]\left( \frac{\phi^2(x)}{\tilde w (x)} \right) \frac{\dx\dy}{|x-y^\lambda|^{n+sp}} 
 \\
 &\geq   -C\int_{B_{R}^c \cap \Sigma_\lambda} \int_{B_{8r}} \left[ \frac{1}{|x-y|^{n+sp}} -\frac{1}{|x-y^\lambda|^{n+sp}} \right] \left[ w(x) \beta(x,y) \right] \mathbbm{1}_{\{w(x)> w(y)\}}\left( \frac{\phi^2(x)}{\tilde w (x)} \right) \dx\dy
\\
&\qquad   -C\norm{u}^{p-2}_{\infty}\int_{B_{R}^c \cap \Sigma_\lambda} \int_{B_{8r}} w(x)  \left( \frac{\phi^2(x)}{\tilde w (x)} \right) \frac{\dx\dy}{|x-y^\lambda|^{n+sp}} 
\\
 &\geq   -C\int_{B^c_{R}} \int_{B_{8r}} \left[ \frac{1}{|x-y|^{n+sp}}  \right] \left[ w(x) \beta(x,y) \right] \left( \frac{\phi^2(x)}{\tilde w (x)} \right) \dx\dy- Cr^n
 \\
 &\geq -Cr^{n-2\alpha},
\end{align*}
where the constant $C$ would depend on the local Lipschitz norm of $u, u_\lambda$ in $B_R$, the $L^\infty$ norm of $u$ and $\dist(B_R, T_\lambda)$. 
Thus, combining the above, we would have $I_2 \geq -C r^{n-2\alpha}$. A similar calculation will also give $I_3 \geq -Cr^{n-2\alpha}$. Therefore, we have recovered \eqref{EL2.6G}. Define $\widehat{w} := \left[ \min \{ \log\frac{1}{2\delta}, \log \frac{k+\varepsilon }{\tilde w} \} \right]_+$. Note that $\tilde w >0$ in $B_{6r}$. Therefore, $\widehat{w}$ is well-defined. 
Now we continue along the lines of Lemma \ref{L2.6} to conclude the proof. 
\end{proof}

Next lemma should be compared with Lemma~\ref{L2.7} and the proof of Theorem~\ref{strongcomp}.
\begin{lem}\label{L3.3}
Let $p>2$, $\frac{sp}{p-2} >1$, and $\lambda$ be such that the conditions of Lemma~\ref{L3.1} hold. Also, assume that $u\in C^1(\Omega)$, $|\grad u(x_0)|\neq 0$ and choose $R$ so that \eqref{cone-anti} 
and Lemma~\ref{L3.1} hold. Then there exists $r>0$ satisfying $\max \{ r^{(1-\frac{p-2}{sp})(p-1)}, 16r \} <R$ such that $u_\lambda=u$ in $B_r(x_0)$.
\end{lem}

\begin{proof}
To apply the proof of Lemma~\ref{L2.7}, we show that if there exist $\sigma \in (0,1], k\in (0, 1)$, such that $w =\bu_\lambda-\bu$ satisfies
\begin{equation}\label{EL3.3A}
|B_r \cap \{ w \geq k\}| \geq \sigma |B_r|
\end{equation} 
for some $r$ satisfying $\max \{ r^{(1-\frac{p-2}{sp})(p-1)}, 16r \} <R$, then there exists a constant $\delta \in (0, \frac{1}{4})$ such that 
\begin{equation}\label{EL3.3B}
\inf_{B_{4r}} (u_\lambda-u)=\inf_{B_{4r}} w \geq \delta k.
\end{equation}
Once \eqref{EL3.3B} is established, we can conclude the lemma arguing as in Theorem~\ref{strongcomp}.

Let $\phi$ be a smooth cutoff function supported in $B_\rho$ for some $r\leq \rho \leq 6r$ and consider the test function $w_- :=(\ell - w)_+$, for some $\ell \in (\delta k, 2 \delta k)$. Then we have,
\begin{align*}
&\int_{\Rn} \int_{\Rn}[J_p(\bu(x) - \bu(y))- J_p(\bu_\lambda(x)- \bu_\lambda(y))](w_-x)\phi^2(x) - w_-(y) \phi^2(y))\frac{\dx\dy}{|x-y|^{n+sp}}
\leq C \ell^2 |B_\rho \cap \{ w < \ell \}|,
\end{align*}
where the right-hand side is computed as in Lemma~\ref{L2.7}. Split the left-hand side as
\begin{align*}
 & \underbrace{\int_{B_{\rho}} \int_{B_{\rho}} \left[ J_p(\bu(x) - \bu(y)) - J_p(\bu_\lambda(x) - \bu_\lambda(y)) \right](w_-(x)\phi^2(x) - w_- (y) \phi^2(y)) \frac{\dx\dy}{|x-y|^{n+sp}}}_{I_1} 
\\
&\quad \underbrace{+ \int_{B_{\rho}} w_-(x) \phi^2(x) dx\int_{B_{\rho}^c} \left[ J_p(\bu(x) - \bu(y)) - J_p(\bu_\lambda(x) - \bu_\lambda(y)) \right] \frac{\dy}{|x-y|^{n+sp}} }_{I_2}
\\
&\qquad \underbrace{-  \int_{B_{\rho}} w_-(y) \phi^2(y) dy\int_{B_{\rho}^c} \left[ J_p(\bu(x) - \bu(y)) - J_p(\bu_\lambda(x) - \bu_\lambda(y)) \right] \frac{\dx}{|x-y|^{n+sp}}  }_{I_3},
\end{align*}
to obtain
\begin{equation}\label{EL3.3C}
I_1+ I_2 + I_3 \leq C \ell^2 |B_\rho \cap \{ w < \ell \}|.
\end{equation}
We mainly need to focus in estimating $I_2$. For $x \in B_\rho$ and $w_-(x) > 0 \Leftrightarrow  \bu_{\lambda}(x) < \ell + \bu(x)$, using Lemma~\ref{L3.1}(ii)-(iv), we see that
\begin{align*}
&\int_{B_{\rho}^c} \left[ J_p(\bu(x) - \bu(y)) - J_p(\bu_\lambda(x) - \bu_\lambda(y)) \right] \frac{\dy}{|x-y|^{n+sp}}
\\
&= \int\limits_{B_{R}^c \cap \Sigma_\lambda} \left[ J_p(\bu(x) - \bu(y)) - J_p(\bu_\lambda(x) - \bu_\lambda(y)) \right] \frac{\dy}{|x-y|^{n+sp}} 
\\
&\quad + \int\limits_{\Sigma_\lambda^c \setminus B_R^\lambda} \left[ J_p(\bu(x) - \bu(y)) - J_p(\bu_\lambda(x) - \bu_\lambda(y)) \right] \frac{\dy}{|x-y|^{n+sp}}
\\
&\qquad +\int\limits_{B_{R}^\lambda} \left[ J_p(\bu(x) - \bu(y)) - J_p(\bu_\lambda(x) - \bu_\lambda(y)) \right] \frac{\dy}{|x-y|^{n+sp}} 
\\
&\qquad \quad+\int\limits_{B_{R} \setminus B_\rho} \left[ J_p(\bu(x) - \bu(y)) - J_p(\bu_\lambda(x) - \bu_\lambda(y)) \right] \frac{\dy}{|x-y|^{n+sp}}
\\
&\geq \int\limits_{B_{R}^c \cap \Sigma_\lambda} \left[ J_p(\bu(x) - \bu(y)) - J_p(\ell + \bu(x) - \bu_\lambda(y)) \right] \frac{\dy}{|x-y|^{n+sp}} 
\\
&\quad + \int\limits_{B_{R}^c \cap \Sigma_\lambda} \left[ J_p(\bu(x) - \bu_\lambda(y)) - J_p(\ell + \bu(x) - \bu(y)) \right] \frac{\dy}{|x-y^\lambda|^{n+sp}}
\\
&\qquad +\int\limits_{B_{R}^\lambda} \left[ J_p(\bu(x) - \bu(y)) - J_p(\ell + \bu(x) - \bu(y)) \right] \frac{\dy}{|x-y|^{n+sp}} 
\\
&\qquad \quad+\int\limits_{B_{R} \setminus B_\rho} \left[ J_p(\bu(x) - \bu(y)) - J_p(\ell + \bu(x) - \bu(y)) \right] \frac{\dy}{|x-y|^{n+sp}}
\\
&\geq \int\limits_{B_{R}^c \cap \Sigma_\lambda} \left[ J_p(\bu(x) - \bu(y)) - J_p(\ell + \bu(x) - \bu_\lambda(y)) \right] \left(\frac{1}{|x-y|^{n+sp}} -\frac{1}{|x-y^\lambda|^{n+sp}} \right) \dy
\\
&\quad + \int\limits_{B_{R}^c \cap \Sigma_\lambda} \Bigl[ J_p(\bu(x) - \bu_\lambda(y))- J_p(\ell + \bu(x) - \bu_\lambda(y))
\\
&\qquad + J_p(\bu(x) - \bu(y))- J_p(\ell + \bu(x) - \bu(y)) \Bigr] \frac{\dy}{|x-y^\lambda|^{n+sp}}
\\
&\qquad +\int\limits_{B_{R}^\lambda} \left[ J_p(\bu(x) - \bu(y)) - J_p(\ell + \bu(x) - \bu(y)) \right] \frac{\dy}{|x-y|^{n+sp}} 
\\
&\qquad \quad+\int\limits_{B_{R} \setminus B_\rho} \left[ J_p(\bu(x) - \bu(y)) - J_p(\ell + \bu(x) - \bu(y)) \right] \frac{\dy}{|x-y|^{n+sp}}
\\
&\stackrel{\eqref{EL2.6D}}{\geq} -C \ell \Biggl[ \int\limits_{B_{R}^c \cap \Sigma_\lambda}  \left[\frac{1}{|x-y|^{n+sp}} -\frac{1}{{|x-y^{\lambda}|}^{n+sp}} \right] \dy
 + \int\limits_{B_{R}^c \cap \Sigma_\lambda}  \frac{\dy}{|x-y^\lambda|^{n+sp}}
\\
&\qquad +\int\limits_{B_{R}^\lambda}  \frac{\dy}{|x-y|^{n+sp}} + \int\limits_{B_{R} \setminus B_\rho}  \frac{\dy}{|x-y|^{n+sp}} \Biggr]
\\
&\geq -C \ell \int\limits_{B_\rho^c}\frac{\dy}{|x-y|^{n+sp}}.
\end{align*}
Therefore,
$$I_2 \geq - C\ell^2 \int_{B_\rho\cap\{w<\ell\}} \phi^2 (x) \int_{B_\rho^c} \frac{\dy}{|x-y|^{n+sp}}
\geq -C \ell^2 \int_{B_\rho\cap\{w<\ell\}} \phi^2 (x) \sup_{x \in \supp \phi} \left( \int_{B_\rho^c} \frac{\dy}{|x-y|^{n+sp}}\right)\dx.
$$
We can estimate $I_3$ similarly. Therefore, we obtain 
$$I_2 + I_3 \geq -C \ell^2 \int_{B_\rho\cap\{w<\ell\}} \phi^2 (x) \sup_{x \in \supp \phi} \left( \int_{B_\rho^c} \frac{\dy}{|x-y|^{n+sp}}\right)\dx.$$ 
Estimate for $I_1$ is the same as in Lemma \ref{L2.7}, and hence we recover \eqref{EL2.7F} from \eqref{EL3.3C}. Rest of the proof follows that of Lemma~\ref{L2.7} verbatim.
Thus, from \eqref{EL3.3A} and Lemma~\ref{L3.2}, we obtain \eqref{EL3.3B}.

Now we can argue as in the first proof of Theorem~\ref{strongcomp} to conclude the result.
\end{proof}

We now conclude this section with the following strong comparison principle.
\begin{thm}\label{T4.4}
Suppose $p>2$.
Let $\lambda\in\R$ be such that $\Omega\cap \Sigma_\lambda\neq\emptyset$ and $(\Omega\cap \Sigma_\lambda)^\lambda\subset\Omega$.
Let $u\in W^{s, p}_{\rm loc}\cap L^\infty(\Rn)$ be a solution to \eqref{main-anti} and $ u_\lambda-u\gneq 0$ in $\Sigma_\lambda$. Then the following hold.
\begin{itemize}
\item[(i)] If $\frac{sp}{p-2}>1$ and $u\in C^1(\Omega)$, then $u_\lambda-u$ can not vanish in $(\Omega\cap \Sigma_\lambda)\setminus\mathcal{Z}_u$.
\item[(ii)] If $\frac{sp}{p-1}< 1$, then $u_\lambda-u>0$ in $\Omega\cap \Sigma_\lambda$.
\end{itemize}
\end{thm}

\begin{proof}
First we consider (i). Suppose, to the contrary, that there exists $x_0\in (\Omega\cap \Sigma_\lambda)\setminus\mathcal{Z}_u$ satisfying
$u_\lambda(x_0)-u(x_0)=0$. Using Lemma~\ref{L3.3}, we find a ball $B_r(x_0)\Subset \Omega\cap\Sigma_\lambda$ so that
$u_\lambda-u=0$ in $B_r(x_0)$. Let $\phi\neq 0$ be a smooth nonnegative function with support in $B_{r}(x_0)$. Using \eqref{main-anti}
we write
$$\int_{\Rn} \int_{\Rn}[J_p(u(x) - u(y))- J_p(u_\lambda(x)- u_\lambda(y))](\phi(x) - \phi(y))\frac{\dx\dy}{|x-y|^{n+sp}} = \int_{\Rn} (f(u(x)) -f(u_\lambda(x)) )\phi(x) \dx =0.$$
On the other hand
\begin{align*}
{\rm LHS}
&= \int_{B_r(x_0)} \int_{B_r(x_0)}[J_p(u(x) - u(y))- J_p(u_\lambda(x)- u_\lambda(y))](\phi(x) - \phi(y))\frac{\dx\dy}{|x-y|^{n+sp}} 
\\ &\quad + \int_{B^c_r(x_0)} \int_{B_r(x_0)}[J_p(u(x) - u(y))- J_p(u_\lambda(x)- u_\lambda(y))]\phi(x) \frac{\dx\dy}{|x-y|^{n+sp}} 
\\ &\qquad - \int_{B_r(x_0)} \int_{B^c_r(x_0)}[J_p(u(x) - u(y))- J_p(u_\lambda(x)- u_\lambda(y))]\phi(y) \frac{\dx\dy}{|x-y|^{n+sp}} 
\\
&=  \underbrace{\int_{B^c_r(x_0)\setminus B^\lambda_r(x_0)} \int_{B_r(x_0)}[J_p(u(x) - u(y))- J_p(u(x)- u_\lambda(y))]\phi(x) \frac{\dx\dy}{|x-y|^{n+sp}}}_{\mathcal{A}} 
\\ &\quad - \underbrace{\int_{B_r(x_0)} \int_{B^c_r(x_0)\setminus B^\lambda_r(x_0)}[J_p(u(x) - u(y))- J_p(u_\lambda(x)- u(y))]\phi(y) \frac{\dx\dy}{|x-y|^{n+sp}} }_{\mathcal{B}},
\end{align*} 
using the fact $u_\lambda=u$ in $B^\lambda_r(x_0)$. 
We only compute $\mathcal{A}$ as $\mathcal{A}=-\mathcal{B}$.
\begin{align*}
\mathcal{A}&=  \int_{B^c_r(x_0)\setminus B^\lambda_r(x_0)} \int_{B_r(x_0)}[J_p(u(x) - u(y))- J_p(u(x)- u_\lambda(y))]\phi(x) \frac{\dx\dy}{|x-y|^{n+sp}}
\\
&=\int_{B^c_r(x_0) \cap \Sigma_\lambda} \int_{B_r(x_0)}[J_p(u(x) - u(y))- J_p(u(x)- u_\lambda(y))]\phi(x) \frac{\dx\dy}{|x-y|^{n+sp}}
\\&\quad + \int_{(B^c_r(x_0) \cap \Sigma_\lambda)^\lambda} \int_{B_r(x_0)}[J_p(u(x) - u(y))- J_p(u(x)- u_\lambda(y))]\phi(x) \frac{\dx\dy}{|x-y|^{n+sp}}
\\
&\geq \int_{B^c_r(x_0) \cap \Sigma_\lambda} \int_{B_r(x_0)}[J_p(u(x) - u(y))- J_p(u(x)- u_\lambda(y))]\phi(x) \left[ \frac{1}{|x-y|^{n+sp}} - \frac{1}{|x-y^\lambda|^{n+sp}} \right] \dx\dy
\\
&\geq 0.
\end{align*}
Therefore, we must have $u_\lambda=u$ in $B^c_r(x_0) \cap \Sigma_\lambda$, implying $u_\lambda-u=0$ in $\Rn$. But this is a contradiction.
Hence $u_\lambda-u$ cannot vanish in $(\Omega\cap \Sigma_\lambda)\setminus\mathcal{Z}_u$.

Now we consider (ii). From Lemma~\ref{PL2.5} we see that $(-\Delta_p)^s u_\lambda$ and $(-\Delta_p)^s u$ are classically defined. Therefore,
if $u_\lambda(x_0)=u(x_0)$ for some $x_0\in \Omega\cap\Sigma_\lambda$, we must have
\begin{align*}
0&=(-\Delta_p)^s u_\lambda(x_0)-(-\Delta_p)^s u(x_0)
\\
&=\int_{\Sigma_\lambda} [J_p(u_\lambda(x_0)-u_\lambda(y))-J_p(u(x_0)-u(y))]\frac{\dy}{|x_0-y|^{n+sp}}
\\
&\quad + \int_{\Sigma_\lambda} [J_p(u_\lambda(x_0)-u(y))-J_p(u(x_0)-u_\lambda(y))]\frac{\dy}{|x_0-y^\lambda|^{n+sp}}
\\
&=\int_{\Sigma_\lambda} [J_p(u_\lambda(x_0)-u_\lambda(y))-J_p(u(x_0)-u(y))]\left(\frac{1}{|x_0-y|^{n+sp}}-\frac{1}{|x_0-y^\lambda|^{n+sp}}\right)\dy
\\
&\quad + \int_{\Sigma_\lambda} \Bigl[\underbrace{J_p(u_\lambda(x_0)-u(y))-J_p(u(x_0)-u(y))}_{= 0}
\\
&\hspace{8em} + \underbrace{J_p(u_\lambda(x_0)-u_\lambda(y))-J_p(u(x_0)-u_\lambda(y))}_{= 0}\Bigr]\frac{\dy}{|x_0-y^\lambda|^{n+sp}}
\\
&\leq 0.
\end{align*}
Thus, we must have $u_\lambda-u=0$ in $\Sigma_\lambda$, which is a contradiction. This proves $u_\lambda-u>0$ in $\Omega\cap\Sigma_\lambda$.
\end{proof}
\section{Symmetry of solution using moving plane method}\label{S-moving}
In this section we prove symmetry of solution for the Dirichlet problem. The rely on the method of moving plane but since we do not have appropriate 
narrow domain maximum principle, we need to modify the standard tools accordingly. Let $\Omega$ be a bounded $C^{1,1}$ domain. Let
$u\in \bX(\Omega)\cap C(\Rn)$ be a weak solution to
\begin{equation}\label{main-moving}
\begin{split}
(-\Delta_p)^s u & =f(u) \quad \text{in}\; \Omega,
\\
u&>0 \quad \text{in}\; \Omega,
\\
u&=0 \quad \text{in}\; \Omega^c.
\end{split}
\end{equation}
Here
$$\bX(\Omega)=\{v\in W^{s, p}(\Rn)\; :\; v=0\quad \text{in}\; \Omega^c\},$$
and $f:\R\to\R$ is locally Lipschitz.
We also clarify that, by a {\it weak} solution of \eqref{main-moving}, we mean that
\[
\int_{\Rn}\int_{\Rn} J_p(u(x)-u(y))(\phi(x)-\phi(y))
\frac{\dx\dy}{|x-y|^{n+sp}}
=
\int_{\Rn} f(u(x))\phi(x)\,\dx
\]
for all $\phi\in \bX(\Omega)$. 
This notion is admittedly slightly stronger than that of a local weak solution. The main reason for adopting this definition is that it allows one to apply the weak comparison principle, which in turn yields the  Hopf's lemma.

Now consider a unit vector $e$ which without any loss of generality, we assume to be $(1, 0, \ldots, 0)$.
For
$\lambda\in\R$, we introduce the following notation.
\begin{gather*}
\Sigma_\lambda=\{x\in\Rn\; : \; x\cdot e<\lambda\}, \quad T_\lambda=\{x\in\Rn: \; x\cdot e=\lambda\}=\partial \Sigma_\lambda, 
\\ 
x^\lambda=(2\lambda-x_1, x_2, \ldots, x_n), \quad u_\lambda(x)=u(x^\lambda),\quad 
A^\lambda=\{x\in\Rn\; :\; x^\lambda\in A\}.
\end{gather*}

Our main result of this section is the following.
\begin{thm}\label{T5.1}
Let $p>2$. Let $\Omega$ be a bounded, strictly convex domain with $C^{1,1}$ boundary, symmetric with respect to 
$T_0=\{x_1=0\}$. Let
$$
u\in \mathbb{X}_0(\Omega)\cap C^1(\Omega\setminus T_0)\cap C(\Rn)
$$
be a local weak solution to \eqref{main-moving}. Assume further that $f$ is locally Lipschitz and satisfies $f(0)\geq 0$. Then $u$ is symmetric with respect to $T_0$, and the following assertions hold:
\begin{itemize}
\item[(i)] If $\frac{sp}{p-1}<1$, then $u$ is strictly increasing in the $x_1$-direction in $\Omega\cap \Sigma_0$.

\item[(ii)] If $\frac{sp}{p-2}>1$, then $u$ is strictly increasing in the $x_1$-direction in $(\Omega\cap \Sigma_0)\setminus \mathcal{Z}_u$ in the following sense: for $z_1<z_2$ such that
$$
(z_i,z')\in \Omega\cap \Sigma_0,\; i=1,2,
$$
and the line segment joining $(z_1,z')$ or $(z_2,z')$ does not lie in $\mathcal{Z}_u$, we have
$$
u(z_1,z')<u(z_2,z').
$$
\end{itemize}
\end{thm}

As an immediate consequence of Theorem~\ref{T5.1}, we obtain the following corollary.

\begin{cor}
Let $p>2$, and let $B$ be a ball centered at the origin. Suppose that
\[
u\in \bX(B)\cap C^1(B\setminus\{0\})\cap C(\Rn)
\]
is a weak solution to
\begin{equation*}
\begin{split}
(-\Delta_p)^s u &= f(u) \quad \text{in }\; B,
\\
u& >0 \quad  \text{in}\; B,
\\
u&=0 \quad \text{in }\; B^c,
\end{split}
\end{equation*}
where $f$ is locally Lipschitz and satisfies $f(0)\geq 0$. Then $u$ is radially symmetric and strictly decreasing along the radial direction in the sense described in Theorem~\ref{T5.1}.
\end{cor}

We also record the following result, which supports the regularity assumptions in Theorem~\ref{T5.1}. This result essentially follows from \cite{GJS25} with only minor modifications.

\begin{thm}\label{T5.3}
Let $p\in [2,\frac{2}{1-s})$, and let $u$ be a local weak solution to
$$
(-\Delta_p)^s u=f(x) \quad \text{in } \Omega,
$$
where $f$ is locally Lipschitz. Then $
u\in C^{1,\alpha}_{\rm loc}(\Omega)
$
for some $\alpha>0$.
In particular, if $p\in [2,\frac{2}{1-s})$, then every local weak solution of
\[
(-\Delta_p)^s u=f(u) \quad \text{in } \Omega,
\]
with $f$ locally Lipschitz, belongs to $C^{1,\alpha}_{\rm loc}(\Omega)$.
\end{thm}
For a sketch of the proof see Appendix~\ref{Appen}.
\medskip

The remaining part of this section is devoted to the proof of Theorem~\ref{T5.1}. We define
 $\lambda_\circ=\inf\{\lambda\in\R\; :\; \Omega\cap\Sigma_\lambda\neq\emptyset\}$. Clearly, $-\infty<\lambda_\circ<0$. Also, let
\begin{align*}
 w_\lambda(x)&=u_\lambda(x)-u(x)=u(x^\lambda)-u(x)\quad x\in \Rn,
 \\
 \Omega_\lambda &=\Omega\cap\Sigma_\lambda.
\end{align*}
We shall make use of the following Hopf's lemma, borrowed from \cite[Theorem~2.6]{IMP23}.
\begin{lem}\label{Hopf}
There exists a constant $\kappa>0$ satisfying 
$$\inf_{x\in \Omega} \frac{u(x)}{(\dist(x, \Omega^c))^s}\geq \kappa.$$
\end{lem}
\begin{proof}
Since $f(0)\geq 0$, we have 
$$(-\Delta_p)^s u - f(u)\geq -f(0)\quad \text{in}\; \Omega.$$
Now the result follows from \cite[Theorem~2.6]{IMP23}. It is important to point that the Lemma in \cite{IMP23}
is stated under the condition $n>sp$, but the proof there does not require this condition.
\end{proof}

We also require the following estimate
\begin{lem}\label{L5.5}
Suppose $p>2$, and $\cone_{x_0}$ be a (unbounded) cone with the vertex at $x_0$. For $\delta>0$ we define $\cone_{x_0}(\delta)=\cone_{x_0}\cap B_\delta(x_0)$.
 Then, for any $\delta>0$, we have 
 $$\int_{\cone_{x_0}(\delta)} \dfrac{(\dist(y, \partial \cone_{x_0}))^{s(p-2)} }{|x_0 - y|^{n+sp}} \dy = +\infty.$$
\end{lem}

\begin{proof}
Without loss of generality, we assume that the cone has unit height (that is, $\delta=1$), $x_0 = 0$ and we denote $\cone_{x_0}$ by $\cone$. For any point $y \in \cone \setminus \{0\}$, we write $y = r\omega$, where $r = |y| $ and $\omega = \frac{y}{|y|} \in \mathbb{S}^{n-1}$.

For any $z \in \partial \cone$ and $r > 0$, the scaled point $z' = \frac{z}{r}$  belongs to $\partial\cone$. Substituting $y = r\omega$ and $z = rz'$, we obtain:
$$\text{dist}(r\omega, \partial \cone) = \inf_{z' \in \partial  \cone} |r\omega - rz'| = r \inf_{z' \in \partial \cone} |\omega - z'| = r\, \dist (\omega, \partial\cone).$$
Therefore, the integral takes the form
\begin{align*}
{\rm LHS}&= \int_0^1 \int_{\cone \cap \mathbb{S}^{n-1}} \frac{r^{s(p-2)} (\dist (\omega, \partial \cone))^{s(p-2)}}{r^{n+sp}} r^{n-1} \dr\,  {\rm d} \sigma(\omega) 
 \\
&=\left( \int_0^1 r^{-2s-1} \, \dr \right) \left( \int_{\cone \cap \mathbb{S}^{n-1}} (\dist (\omega, \partial \cone))^{s(p-2)} \, {\rm d} \sigma (\omega) \right),
 \end{align*}
 where $\sigma$ is the $(n-1)$-dimensional Hausdorff measure on $\mathbb{S}^{n-1}$. The first integral is unbounded, and since $\cone$ has a non-empty interior, the second integral is strictly positive. Hence the result.
\end{proof}

\noindent {\bf Step 1.} (Starting the moving plane) We show that for $\lambda \in (\lambda_\circ,0)$, sufficiently close to $\lambda_\circ$, we have
\begin{equation}\label{E5.2}
w_\lambda \geq 0\quad  \text{ in }\;  \Omega_\lambda.
\end{equation}
Suppose otherwise, then for $\cM_\lambda:=\{x\in \Sigma_\lambda\; :\; w_\lambda(x)=\min_{\Sigma_\lambda} w_\lambda<0\}$, we must have $\cM_\lambda\Subset \Omega_\lambda$. This holds since
$w_\lambda\geq 0$ in $(\Sigma_\lambda\setminus \Omega_\lambda)\cup T_\lambda$. Fix $\widehat\Omega_\lambda$ satisfying $\cM_\lambda\Subset\widehat\Omega_\lambda\Subset \Omega_\lambda$.
Let $(u_\lambda)_\varepsilon$ and $u^\varepsilon$ be the inf and sup convolution of $u_\lambda$ and $u$, respectively. Denote by $w^\varepsilon_\lambda=(u_\lambda)_\varepsilon-u^\varepsilon$. Since
$w^\varepsilon_\lambda\to w_\lambda$ as $\varepsilon\to 0$, uniformly on compacts, for points $x_{\varepsilon}\in\Argmin_{\bar\Omega_\lambda} w^\varepsilon_\lambda \to \cM_\lambda$ as $\varepsilon\to 0$, up to a subsequence.
Thus, we can find $\varepsilon_0>0$ such that for all $\varepsilon\leq \varepsilon_0$ we have $x_\varepsilon\in \widehat\Omega_\lambda$. From Lemmas~\ref{PL2.2}-\ref{PL2.4} we then have
$$(-\Delta_p)^s (u_\lambda)_\varepsilon(x_\varepsilon)\geq f((u_\lambda)_\varepsilon(x_\varepsilon))+\eta_\varepsilon(x_\varepsilon)
\quad \text{and}\quad
(-\Delta_p)^s u^\varepsilon(x_\varepsilon)\leq f(u^\varepsilon(x_\varepsilon))+\tilde\eta_\varepsilon(x_\varepsilon).$$
Let $r>0$ be such that $B_r(x_\varepsilon)\subset \Omega_\lambda$ for all $\varepsilon\leq \varepsilon_0$, otherwise, we can set $\varepsilon_0$ smaller. Since the operators are calculated classically
at $x_\varepsilon$, we find
\begin{align}\label{E5.3}
(-\Delta_p)^s (u_\lambda)_\varepsilon(x_\varepsilon) - (-\Delta_p)^s u^\varepsilon(x_\varepsilon) &\geq  f((u_\lambda)_\varepsilon(x_\varepsilon)) - f(u^\varepsilon(x_\varepsilon)) + \eta_\varepsilon(x_\varepsilon)
-\tilde\eta_\varepsilon(x_\varepsilon)\nonumber
\\
&\geq -L |w^\varepsilon_\lambda(x_\varepsilon)|  + \eta_\varepsilon(x_\varepsilon) - \tilde\eta_\varepsilon(x_\varepsilon),
\end{align}
where $L$ denotes the local Lipschitz constant. On the other hand, the left-hand side above is computed as 
\begin{align*}
{\rm LHS}
&= \int_{\Sigma_\lambda} [J_p((u_\lambda)_\varepsilon(x_\varepsilon) - (u_\lambda)_\varepsilon(y)) - J_p (u^\varepsilon(x_\varepsilon) - u^\varepsilon(y))]
 \left(\frac{1}{|x_\varepsilon - y|^{n+sp}} - \frac{1}{|x_\varepsilon - y^\lambda|^{n+sp}} \right)\dy
\\
&\quad
+ \int_{\Sigma_\lambda} \Bigl[J_p((u_\lambda)_\varepsilon(x_\varepsilon) - (u_\lambda)_\varepsilon(y)) - J_p (u^\varepsilon(x_\varepsilon) - u^\varepsilon(y^\lambda))
\\
&\hspace{8em} + J_p((u_\lambda)_\varepsilon(x_\varepsilon) - (u_\lambda)_\varepsilon(y^\lambda)) - J_p (u^\varepsilon(x_\varepsilon) - u^\varepsilon(y))\Bigr]	
\frac{\dy}{|x_\varepsilon - y^\lambda|^{n+sp}}
\\
&\leq\int_{\Sigma_\lambda\setminus B_r(x_\varepsilon)} [J_p((u_\lambda)_\varepsilon(x_\varepsilon) - (u_\lambda)_\varepsilon(y)) - J_p (u^\varepsilon(x_\varepsilon) - u^\varepsilon(y))]
 \left(\frac{1}{|x_\varepsilon - y|^{n+sp}} - \frac{1}{|x_\varepsilon - y^\lambda|^{n+sp}} \right)\dy
\\
&\quad
+ \int_{\Sigma_\lambda} \Bigl[J_p((u_\lambda)_\varepsilon(x_\varepsilon) - (u_\lambda)_\varepsilon(y)) - J_p (u^\varepsilon(x_\varepsilon) - u^\varepsilon(y^\lambda))
\\
&\hspace{8em} + J_p((u_\lambda)_\varepsilon(x_\varepsilon) - (u_\lambda)_\varepsilon(y^\lambda)) - J_p (u^\varepsilon(x_\varepsilon) - u^\varepsilon(y))\Bigr]	
\frac{\dy}{|x_\varepsilon - y^\lambda|^{n+sp}}
\end{align*}
since in $B_r(x_\varepsilon)\subset\Omega_\lambda$ we have $(u_\lambda)_\varepsilon(x_\varepsilon)-u^\varepsilon(x_\varepsilon)\leq (u_\lambda)_\varepsilon(y)-u^\varepsilon(y)$, implying
$$
J_p((u_\lambda)_\varepsilon(x_\varepsilon) - (u_\lambda)_\varepsilon(y)) - J_p (u^\varepsilon(x_\varepsilon) - u^\varepsilon(y))\leq 0.
$$
Now let $\varepsilon\to 0$, and using $x_\varepsilon\to x_0\in\cM_\lambda$ (up to a subsequence) in \eqref{E5.3} we obtain
\begin{align}\label{E5.4}
 L w_\lambda(x_0)
& \leq\int_{\Sigma_\lambda\setminus B_r(x_0)} [J_p(u_\lambda(x_0) - u_\lambda(y)) - J_p (u(x_0) - u(y))]
 \left(\frac{1}{|x_0 - y|^{n+sp}} - \frac{1}{|x_0 - y^\lambda|^{n+sp}} \right)\dy\nonumber
\\
&\quad
+ \int_{\Sigma_\lambda} \Bigl[J_p(u_\lambda(x_0) - u_\lambda(y)) - J_p (u(x_0) - u_\lambda(y))\nonumber
\\
&\hspace{8em} + J_p(u_\lambda(x_0) - u(y)) - J_p (u(x_0) - u(y))\Bigr]	
\frac{\dy}{|x_0 - y^\lambda|^{n+sp}}\nonumber
\\
&\leq \int_{\Sigma_\lambda} \Bigl[J_p(u_\lambda(x_0) - u_\lambda(y)) - J_p (u(x_0) - u_\lambda(y))\nonumber
\\
&\hspace{8em} + J_p(u_\lambda(x_0) - u(y)) - J_p (u(x_0) - u(y))\Bigr]:=I_2,
\end{align}
where we use the fact $w_\lambda(x_0)\leq w_\lambda(y)$ in $\Sigma_\lambda$ to make the first integral negative.

To compute $I_2$, we use the fundamental theorem of calculus to see that 
\begin{align*}
J_p(u_\lambda(x_0) - u_\lambda(y)) - J_p (u(x_0) - u_\lambda(y)) &= w_\lambda(x_0) (p-1)\int_0^1 |u(x_0) - u_\lambda(y)+t w_\lambda(x_0)|^{p-2}\dt,
\\
J_p(u_\lambda(x_0) - u(y)) - J_p (u(x_0) - u(y))&= w_\lambda(x_0) (p-1)\int_0^1 |u(x_0) - u(y)+t w_\lambda(x_0)|^{p-2}\dt.
\end{align*}
Since $w_\lambda(x_0)<0$ and $|a|^{p-2}+|b|^{p-2}\geq \frac{1}{\max\{1, 2^{p-3}\}}|a-b|^{p-2}$ for $a, b\in\R$ and $p\geq 2$, we obtain
\begin{align*}
I_2\leq  w_\lambda(x_0) \frac{p-1}{\max\{1, 2^{p-3}\}}\int_{\Sigma_\lambda} |w_\lambda(y)|^{p-2}\frac{\dy}{|x_0-y^\lambda|^{n+sp}}.
\end{align*}
Thus, using \eqref{E5.4}, we arrive at
$$ \int_{\Sigma_{\lambda_\circ}} |w_\lambda(y)|^{p-2}\frac{\dy}{|x_0-y^\lambda|^{n+sp}}\leq \int_{\Sigma_\lambda} |w_\lambda(y)|^{p-2}\frac{\dy}{|x_0-y^\lambda|^{n+sp}}\leq \frac{L}{p-1} \max\{1, 2^{p-3}\}.$$
Now letting $\lambda\searrow \lambda_\circ$, we find $x_0\to \tilde{x}\in\partial\Omega$ (up to a subsequence) and apply Fatou's lemma in the above
inequality to see that 
\begin{align*}
\frac{L}{p-1} \max\{1, 2^{p-3}\} &\geq \int_{\Sigma_{\lambda_\circ}} |w_{\lambda_\circ}(y)|^{p-2}\frac{\dy}{|\tilde{x}-y^{\lambda_\circ}|^{n+sp}}
\\
&= \int_{\Sigma_{\lambda_\circ}} |u_{\lambda_\circ}(y)|^{p-2}\frac{\dy}{|\tilde{x}-y^{\lambda_\circ}|^{n+sp}}
\\
&= \int_{\Sigma^c_{\lambda_\circ}} |u(y)|^{p-2}\frac{\dy}{|\tilde{x}-y|^{n+sp}}
\\
&\geq \kappa \int_{\Omega} (\dist(y,\Omega^c))^{s(p-2)}\frac{\dy}{|\tilde{x}-y|^{n+sp}},
\end{align*}
where in the last line we used Lemma~\ref{Hopf}.
Since $\Omega$ has $C^{1,1}$ boundary, it satisfies interior sphere condition, and therefore, for some cone $\cone_{\tilde x}$ with vertex at $\tilde{x}$,
and some $\delta>0$ we would have $(\cone_{\tilde x}\cap B_\delta(\tilde{x}))\setminus\{x\}\subset \Omega$. For $\delta$ sufficiently small, it is also evident that $\dist(y,\Omega^c)\geq \dist(y, \partial\cone_{\tilde{x}})$ for all $y\in \cone_{\tilde x}\cap B_\delta(\tilde{x})$. Hence, we obtain
$$ \int_{\cone_{\tilde x}\cap B_\delta(\tilde{x})} (\dist(y, \partial\cone_{\tilde{x}}))^{s(p-2)}\frac{\dy}{|\tilde{x}-y|^{n+sp}}\leq C,$$
which contradicts Lemma ~\ref{L5.5}. Thus $w_\lambda\geq 0$ for all $\lambda$ sufficiently close to $\lambda_\circ$, proving \eqref{E5.2}.

Define 
$$\lambda^*=\max\{\lambda\in (\lambda_\circ, 0)\; :\; w_\kappa\geq 0\; \quad \text{for all}\; \kappa\in (\lambda_\circ, \lambda]\}.$$
From Step 1, it follows that $\lambda^*>\lambda_\circ$. To complete the method of moving plane we must show that $\lambda^*=0$.

\medskip
\noindent{\bf Step 2.} (Proving $\lambda^*=0$) Suppose, on the contrary, that $\lambda^*<0$. Then, by definition, there exists
a sequence of $\lambda_k\in (\lambda_\circ, 0)$, satisfying $\lambda_k\searrow \lambda^*$, and
$$\inf_{\Sigma_{\lambda_k}} w_{\lambda_k}<0.$$
Let $\cM_k:=\{x\in \Sigma_{k}\; :\; w_k(x)=\min_{\Sigma_k} w_\lambda<0\}$, where $w_k:=w_{\lambda_k}$ and $\Omega_k=\Omega_{\lambda_k}$.
Following Step 1, we find
 $x_k\in \cM_k$ such that for any $r_k<\frac{1}{2}\dist(\cM_k, \partial\Omega_k)$ we have
 \begin{align}\label{E5.5}
 &f(u_{\lambda_k}(x_k))-f(u(x_k))\nonumber
 \\
&\quad \leq\int_{\Sigma_k\setminus B_{r_k}(x_k)} [J_p(u_\lambda(x_k) - u_\lambda(y)) - J_p (u(x_k) - u(y))]
 \left(\frac{1}{|x_k - y|^{n+sp}} - \frac{1}{|x_k - y^{\lambda_k}|^{n+sp}} \right)\dy\nonumber
\\
&\qquad
+ \int_{\Sigma_k} \Bigl[J_p(u_{\lambda_k}(x_k) - u_{\lambda_k}(y)) - J_p (u(x_k) - u_{\lambda_k}(y))\nonumber
\\
&\hspace{12em} + J_p(u_{\lambda_k}(x_k) - u(y)) - J_p (u(x_k) - u(y))\Bigr]	
\frac{\dy}{|x_k - y^{\lambda_k}|^{n+sp}}\nonumber
\\
&:= I_{1, k} + I_{2,k}.
\end{align}
Since $w_k(x_k)\leq \min\{w_k(y), 0\}$ in $\Sigma_k$, we note that $I_{1, k}\leq 0$ and $I_{2,k}\leq 0$.

We let $k\to\infty$, and without any loss of generality we may assume that $x_k\to x^*\in \bar{\Omega}_{\lambda^*}$ (otherwise, we move to a subsequence). Since $w_{\lambda^*}\geq 0$
in $\Sigma_{\lambda^*}$, it follows that $w_{\lambda^*}(x^*)=0$. To analyze the limit, we consider following cases.
\medskip

Let $x^*\in \bar{\Omega}_{\lambda^*}\setminus T_{\lambda^*}$. Letting $k\to\infty$ in \eqref{E5.5}, we see that the left-hand side goes to $0$. Fix $r>0$ small enough so that $\Sigma_{\lambda^*}\setminus B_r(x^*)\subset \Sigma_k\setminus B_{r_k}(x_k)$ for all large $k$. More precisely, if 
$\dist(\cM_k, \partial\Omega_k)\to 0$, we can take any small $r>0$, whereas if $\liminf_{k\to\infty} \dist(\cM_k, \partial\Omega_k)> 0$,
we can set $r_k$ to constant satisfying $r=2 r_k< \frac{1}{2}\dist(\cM_k, \partial\Omega_k)$. Thus, using the fact $u_{\lambda^*}(x^*)=u(x^*)$ together with the Fatou's lemma,
we obtain from \eqref{E5.5} that
$$
\int_{\Sigma_{\lambda^*}\setminus B_{r}(x^*)} [J_p(u(x^*) - u_\lambda(y)) - J_p (u(x^*) - u(y))]
 \left(\frac{1}{|x^* - y|^{n+sp}} - \frac{1}{|x^* - y^{\lambda^*}|^{n+sp}} \right)\geq 0.
$$
Since $|x^* - y|<|x^* - y^{\lambda^*}|$ and $w_{\lambda^*}\gneq 0$  in $\Sigma_{\lambda^*}$, the left-hand side is negative, which leads to a contradiction. Thus $x^*\notin \bar{\Omega}_{\lambda^*}\setminus T_{\lambda^*}$.

\medskip

Next suppose $x^*\in\partial\Omega\cap T_{\lambda^*}$. Note that the 
$${\rm LHS\; of}\;  \; \eqref{E5.5}\geq L w_k(x_k),$$
where $L$ denotes the local Lipschitz constant of $f$, dependent on $\norm{u}_\infty$. Again, as done in Step 1,
$$I_{2, k}\leq w_k(x_k) \frac{p-1}{\max\{1, 2^{p-3}\}}\int_{\Sigma_k} |w_k(y)|^{p-2}\frac{\dy}{|x_k-y^{\lambda_k}|^{n+sp}}.$$
Plugging these in \eqref{E5.5}, letting $k\to\infty$ and applying Fatou's lemma we arrive at
\begin{equation}\label{E5.6}
\int_{\Sigma_{\lambda^*}} |w_{\lambda^*}(y)|^{p-2}\frac{\dy}{|x^*-y^{\lambda^*}|^{n+sp}}
\leq \frac{L}{p-1} \max\{1, 2^{p-3}\},
\end{equation}
using the fact $I_{1,k}\leq 0$. Since $\lambda^*<0$ and $\Omega$ is strictly convex, we can find a small ball $B$ inside 
$(\Omega\setminus\Omega_{\lambda^*})\setminus (\bar\Omega_{\lambda^*})^{\lambda^*}$. Moreover, due to the convexity of the domain,
the cone $\{ t x^* + (1-t)B\; :\; t\in [0, 1]\}$, would lie in $(\Omega\setminus\Omega_{\lambda^*})\setminus (\bar\Omega_{\lambda^*})^{\lambda^*}$ (see Figure ~\ref{fig}).
Hence, for some $\delta>0$ we have 
$$(\cone_{x^*}(\delta))^{\lambda^*}\subset (\Omega\setminus\Omega_{\lambda^*})^{\lambda^*}\setminus \bar\Omega_{\lambda^*}.$$
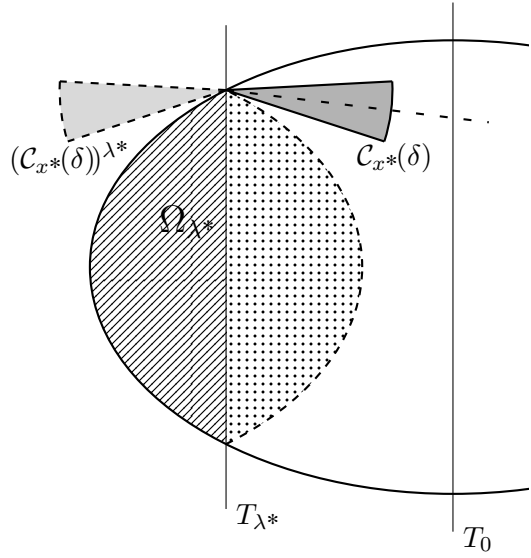
\begin{figure}[h]
    \centering
    \begin{tikzpicture}

\def\Rx{4.8} 
\def\Ry{3.0} 
\def\xT{-3.0} 
\def\refCX{2*\xT} 

\begin{scope}
\clip (-\Rx-0.5, -\Ry-0.5) rectangle (\xT, \Ry+0.5); 
\clip (0,0) ellipse ({\Rx} and {\Ry}); 
\fill[pattern=north east lines] (-\Rx, -\Ry) rectangle (\xT, \Ry);
\end{scope}

\begin{scope}
\clip (\xT, -\Ry-0.5) rectangle (\Rx, \Ry+0.5); 
\clip (\refCX, 0) ellipse ({\Rx} and {\Ry}); 
\fill[pattern=dots] (\xT, -\Ry) rectangle (\Rx, \Ry);
\end{scope}

\pgfmathsetmacro{\yIntersect}{\Ry * sqrt(1 - (\xT*\xT)/(\Rx*\Rx))}
\coordinate (Apex) at (\xT, \yIntersect);

\begin{scope}[shift={(Apex)}]
    \fill[black,opacity=0.3] (0,0) -- (3:2.2) arc (3:-18:2.2) -- cycle;

    \draw[thick] (0,0) -- (3:2.2);
    \draw[thick] (0,0) -- (-18:2.2);
    \draw[thick] (3:2.2) arc (3:-18:2.2);

    \draw[thick, loosely dashed] (0,0) -- (-7:3.5) coordinate (P);
\end{scope}

\begin{scope}[shift={(Apex)}, xscale=-1]
  
    \fill[black,opacity=0.15] (0,0) -- (3:2.2) arc (3:-18:2.2) -- cycle;

    \draw[thick, dashed] (0,0) -- (3:2.2);
    \draw[thick, dashed] (0,0) -- (-18:2.2);
    \draw[thick, dashed] (3:2.2) arc (3:-18:2.2);

\end{scope}
\node[right] at (\xT - 3, \yIntersect -.9) { $(\mathcal{C}_{x^*} \! (\delta)\!)^{\lambda^*}$};


\begin{scope}
\def\clipRight{1.0} 
\clip (-\Rx-0.5, -\Ry-0.5) rectangle (\clipRight, \Ry+0.1);
\draw[thick] (0,0) ellipse ({\Rx} and {\Ry});
\end{scope}

\begin{scope}
\clip (\xT, -\Ry-0.5) rectangle (\Rx, \Ry+0.5); 
\draw[thick, dashed] (\refCX, 0) ellipse ({\Rx} and {\Ry});
\end{scope}

\draw[] (\xT, -\Ry-0.2) -- (\xT, \Ry+0.2);


\node at (-3.5, 0.6) {\Large $\Omega_{\lambda^*}$};

\node[below right] at (\xT, -\Ry) { $T_{\lambda^*}$};

\node[right] at (\xT + 1.6, \yIntersect -.9) { $\mathcal{C}_{x^*} \! (\delta)$};

\draw[thin] (0, -\Ry-.5) -- (0, \Ry+.5); 
\node[below right] at (0, -\Ry-0.3) { $T_{0}$};

\end{tikzpicture} 
    \caption{Visualization of the cone}
    \label{fig}
\end{figure}
Note that $u=0$ in $(\cone_{x^*}(\delta))^{\lambda^*}$. So we can apply Lemmas~\ref{Hopf} and ~\ref{L5.5}, as done in Step 1, to see that the
left-hand side of \eqref{E5.6} is infinite, leading to a contradiction.

\medskip

Thus we are left with the option $x^*\in \Omega\cap T_{\lambda^*}$. Let $\delta_k = \dist (x_k, \partial \Sigma_k)$. It is obvious that $\delta_k\to 0$
and $\dist(\cM_k, \partial\Omega_k)\to 0$, as $k\to\infty$. From \eqref{E5.5} we then have, for $r>0$,
\begin{align}\label{E5.7}
L w_k(x_k)\leq I_{1, k}
&\leq\int_{\Sigma_k\setminus B_{r}(x^*)} [J_p(u_{\lambda_k}(x_k) - u_{\lambda_k}(y)) - J_p (u(x_k) - u(y))]
 \left(\frac{1}{|x_k - y|^{n+sp}} - \frac{1}{|x_k - y^{\lambda_k}|^{n+sp}} \right)\nonumber
 \\
 &=J_{1,k}
\end{align}
for all $k$ large, depending on $r$. Denote by $x_k=(x_{k1}, \ldots, x_{kn})$ and $y_k=(\lambda_k, x_{k2}, \ldots, x_{kn})$. 
Note that $\delta_k=|x_k-y_k|=\lambda_k-x_{k1}$.
Since $x^*\in \Omega$ and $w\in C^{1}(\Omega)$, it follows that
$$
\frac{w_k(x_k)}{\delta_k}=\frac{w_k(x_k)-w_k(y_k)}{\delta_k}=\grad w_k(\xi_k)\cdot \tilde{e}_k=(\grad w_k(\xi_k)-\grad w_k(x_k))\cdot \tilde{e}_k\to 0
$$
as $k\to\infty$, where  $\xi_k$ is a point on the line joining $x_k$ and $y_k$, and
$\tilde{e}_k$ is the unit vector along $x_k-y_k$. To derive a contradiction from \eqref{E5.7} we show that, for $r$ sufficiently small,
\begin{equation}\label{E5.8}
\limsup_{k\to\infty}\frac{1}{\delta_k}J_{1, k}<0.
\end{equation}
Estimate \eqref{E5.8} actually follows from \cite[Theorem~2.3]{CL18}. We add a proof here for the convenience of reading. Applying
mean-value theorem on $g(t)=t^{-\frac{n+sp}{2}}$ we note that
\begin{align*}
 \frac{1}{\delta_k}\left(\frac{1}{|x_k - y|^{n+sp}} - \frac{1}{|x_k - y^{\lambda_k}|^{n+sp}}\right)
 &= 2(n+sp) \frac{(\lambda_k-y_1)}{|\xi(k, y)|^{n+sp+2}},
\end{align*}
where $|x_k-y|\leq \xi(k, y)\leq |x_k-y^{\lambda_k}|$. Since $x^*\in T_0$, we have $\xi(k, y)\to |x^*-y|$ as $k\to\infty$.
Therefore,
\begin{align*}
\lim_{k\to\infty}\frac{1}{\delta_k}J_{1, k}
=2(n+sp)\int_{\Sigma_{\lambda^*}\setminus B_{r}(x^*)} [J_p(u(x^*) - u_{\lambda^*}(y)) - J_p (u(x^*) - u(y))]
\frac{(\lambda^*-y_1)}{|x^*-y|^{n+sp+2}}\dy,
\end{align*}
using the fact $u_{\lambda^*}(x^*)=u(x^*)$. Since $w_{\lambda^*}\gneq 0$ in $\Sigma_{\lambda^*}$, we can choose $r$ small enough so that the above limit becomes negative. This proves \eqref{E5.8}, which contradicts \eqref{E5.7}. Thus, $x^*\notin \Omega\cap T_0$.

Hence we have covered all possibilities of $x^*$ and the contradiction in each case suggests that $\lambda^*$ must be equal to $0$.

\begin{rem}\label{R5.1}
Step 2 above should be compared with Step 2 in the proof of \cite[Theorem~3.1]{CL18}. Observe that the sequence $x_k$ considered in \cite{CL18} converges to a point
$x_0\in \partial\Sigma_{\lambda^*}$,
which may lie in $\partial\Sigma_{\lambda^*}\cap \partial\Omega$. In such a situation, the condition
$$
\nabla w_{\lambda^*}(x_0)=0
$$
(see equation (39) in \cite{CL18}) may fail unless one assumes
$u\in C^1(\overline{\Omega})$,
which itself may be incompatible with the Hopf's lemma. On the other hand, the proof of \cite[Theorem~3.1]{CL18} remains valid when
$x_0\in \partial\Sigma_{\lambda^*}\cap \Omega$,
which is one of the situations considered in Step 2 above.
\end{rem}

Now we can complete the proof of Theorem~\ref{T5.1}.
\begin{proof}[Proof of Theorem~\ref{T5.1}]
In view of Step 2 above, we get $u_\lambda-u\geq 0$ for all $\lambda\in (\lambda_\circ, 0]$. Again, since $u_0$ also satisfies \eqref{main-moving},
we obtain $u_0(x)=u(-x_1, x_2, \ldots, x_n)=u(x_1, x_2, \ldots, x_n)=u(x)$. Moreover, we also see get that $u$ is non-decreasing in  the $x_1$ direction. To investigate the strict monotonicity, first we suppose $sp>p-2$.
Consider $x_0\in\Omega_0\setminus \mathcal{Z}_u$. Since $\mathcal{Z}_u$ is closed relative to $\Omega_0$, we can find a 
ball $B_\delta(x_0)\Subset \Omega_0\setminus \mathcal{Z}_u$. Let $x_1, x_2\in B_\delta(x_0), x_1<x_2$
 be such that the line passing through $x_1, x_2$
contains $x_0$, and $x_2-x_1$ is parallel to the $x_1$-axis. Define $\lambda=\frac{1}{2}(x_1+x_2)<0$. From Step 2 we have
$u_\lambda(x)-u(x)\gneqq 0$ in $\Sigma_\lambda$. Applying Theorem~\ref{T4.4} we see that $u_\lambda-u>0$ in 
$\Omega_\lambda\cap  B_\delta(x_0)$. In particular, $u(x_2)=u_\lambda(x_1)>u(x_1)$. In other words, $u$ is strictly increasing on the line passing 
through $x_0$ belonging to $B_\delta(x_0)$, and parallel to the $x_1$-axis. Hence (ii) follows.

(i) follows from Theorem~\ref{T4.4} and the argument above. This completes the proof.
\end{proof}


\appendix

\section{Proof of Theorem~\ref{T5.3}}\label{Appen}
By exploiting the scaling properties of the equation and \cite[Theorem~1.2]{BS25}, it suffices to consider the case where $\Omega=B_2$, and $u, f$ are Lipschitz in $B_2$.

The proof essentially follows the arguments in \cite{GJS25}, where the authors deal with fractional $p$-harmonic functions. The same reasoning can be extended to the Poisson equation with Lipschitz data. Consider a local weak solution to
\begin{equation}\label{EA1}
-(\Delta_p)^s u=f(x)\quad \text{in}\;\; B_2.
\end{equation}
For any unit vector $e\in\mathbb{S}^n$, we denote $g_e(x)=e\cdot \nabla f$. We also define the operator $\cL_u v(x)$ by
$$
\cL_u v(x) = \text{pv}\int_{\Rn} (v(x)-v(y)) K_u(x, y)\,\dy,
$$
where the kernel is given by
$$
K_u(x, y)=(p-1) \frac{|u(x)-u(y)|^{p-2}}{|x-y|^{n+sp}}.
$$
Note that this kernel is symmetric, that is, $K_u(x, y)=K_u(y, x)$, but may degenerate. Furthermore, by \cite[Lemma~3.2]{GJS25}, if $v \in W^{2, \alpha}(B_1)$ for $\alpha=1-\frac{p(1-s)}{2}$, then $\cL_u v \in W^{2, -\alpha}(B_1)$. In particular, $\cL_u v$ is well-defined in the weak sense provided that $v$ is Lipschitz.

Informally, one can see that $\cL_u (e\cdot\nabla u)=g_e$. The H\"older regularity of $\nabla u$ is then obtained by analyzing this linearized model; see \cite{GJS25}. Since $u$ is not globally Lipschitz, we must introduce a cut-off function to rigorously define the linearized equation. Let $\eta:\Rn\to[0, 1]$ be a radially symmetric cut-off function satisfying $\eta(x)=1$ for $|x|\leq 3/2$ and $\eta(x)=0$ for $|x|\geq 7/4$. For $R>1$, we define
$$
v_e(x)=\eta(x/R) (e\cdot \nabla u).
$$
If $u$ is sufficiently smooth and satisfies \eqref{EA1} in $B_{2R}$, and if $f$ is Lipschitz in $B_{2R}$, then the computations in \cite[Lemma~3.3]{GJS25} reveal that
\begin{equation}\label{EA2}
|\cL_u v_e|\leq C \left(R^{-1-sp} \|u\|^{p-1}_{L^\infty(B_R)}+ [\tail_{sp+1,p-1}(0, R;u)]^{p-1} + \|f\|_{\text{Lip}(B_{2R})} \right)\quad \text{for}\; x\in B_R,
\end{equation}
where $C$ depends only on $s, p$, and $n$. Rather than directly differentiating \eqref{EA1}, we can employ difference quotients and follow the calculations in \cite[Lemma~3.3]{GJS25} to derive \eqref{EA2}. This estimate can be rigorously justified in the weak sense via the regularization method presented in Section 9 of \cite{GJS25} without any modifications.

Let $K_0\in\mathbb{N}$, and define $u_1(x) = 2^{K_0} u (2^{-K_0}x)$. From \eqref{EA1}, it follows that
$$
-(\Delta_p)^s u_1=f_1(x)\quad \text{in}\;\; B_{2^{K_0+1}},
$$
where
\[
f_1(x)= 2^{K_0((p-1)-sp)} f(2^{-K_0}x).
\]
Since $sp>p-2$, choosing a sufficiently large $K_0$ allows us to make $\|f_1\|_{\text{Lip}(B_{2^{K_0+1}})}$ arbitrarily small. This is the key observation required to apply the results of \cite{GJS25}. Keeping this scaling in mind, the proof of \cite[Lemma~4.1]{GJS25} yields the following: given small constants $\mu, r>0$, there exist small constants $\varepsilon_1, \delta>0$ and a large constant $K_0>0$ (all depending on $r, \mu, s, p, n$) such that for any $e\in \mathbb{S}^n$, if
\begin{itemize}
\item[(i)] $u$ satisfies \eqref{EA1} in $B_{2^{K_0+1}}$,
\item[(ii)] $u(0)=0$,
\item[(iii)] $\|\nabla u\|_{L^\infty(B_{2^k})}\leq (1-\delta)^{-k}$ for $k=0, \ldots, K_0$,
\item[(iv)] $\tail_{p-1,sp+1}(0, 2^{K_0}; u) + \|f_1\|_{\text{Lip}(B_{2^{K_0}})}\leq \varepsilon_1$,
\item[(v)] $|\{x\in B_1 : |\nabla u-e|\geq r\}|\geq \mu$,
\end{itemize}
then $e\cdot \nabla u\leq (1-\delta)$ in $B_{\frac{1}{2}}$. 

Since $\delta$ can be chosen small enough to satisfy $(1-\delta)^{-(p-1)}2^{p(1-s)-2}<1$, this conclusion along with a standard scaling argument yields \cite[Lemma~5.1]{GJS25}.
Again, by scaling argument, the complementary case to the one above is given by
$$
|\{x\in B_1 : |\nabla u-e|\leq r\}|\geq (1-\mu).
$$
This is the core content of \cite[Section 6]{GJS25}. Because that section relies on the Ishii-Lions method and most of the computations are analogous to those in \cite{BS25,BT25}, the results continue to hold in this framework. Consequently, the H\"older regularity results from \cite[Section 7]{GJS25} remain valid under the conditions of \cite[Lemma~7.1]{GJS25}. Finally, by combining these results as in \cite[Section 8]{GJS25}, we establish the H\"{o}lder regularity of $\nabla u$ in the ball $B_1$.

To prove the second part of Theorem~\ref{T5.3}, it remains only to show that $u$ is locally Lipschitz. If $\frac{sp}{p-1}>1$, this follows directly from \cite[Theorem~1.1]{BT25}. On the other hand, if $\frac{sp}{p-2}>1$, the result follows by applying \cite[Theorem~1.2]{BS25} along with a standard bootstrapping argument.

\bigskip
\subsection*{Acknowledgement}
A. Biswas was partially supported by 
an ANRF-ARG grant ANRF/ARG/ 2025/000019/MS .

\subsection*{Data Availability}  All data generated or analyzed during this study are included in this published article.
\medskip

\noindent{\bf Declarations}
\subsection*{Conflicts of Interest}  The authors declare that they have no conflict of interest to this work.

\end{document}